\documentclass[11pt]{amsart}

\usepackage{mathptmx}       
\usepackage{helvet}         
\usepackage{courier}        
\usepackage{type1cm}        
\usepackage[T1]{fontenc}

\usepackage{a4wide}
\usepackage[dvips]{epsfig}
\usepackage{amscd}
\usepackage{amssymb}
\usepackage{amsthm}
\usepackage{amsfonts}
\usepackage{amsmath}
\usepackage{comment}
\usepackage{latexsym}
\usepackage{graphicx}
\usepackage[all]{xy}

\usepackage{float}

\usepackage[colorinlistoftodos]{todonotes} 

\usepackage{hyperref}
\hypersetup{
		colorlinks,
    citecolor=black,%
    filecolor=black,%
    linkcolor=black,%
    urlcolor=black,
    pdfstartview={FitH},
    pdftitle={title},
    pdfauthor={},
}

\newtheorem{theorem}{Theorem}
\theoremstyle{plain}

\newtheorem{corollary}[theorem]{Corollary}

\newtheorem{definition}[theorem]{Definition}
\newtheorem{example}[theorem]{Example}

\newtheorem{lemma}[theorem]{Lemma}

\newtheorem{proposition}[theorem]{Proposition}
\newtheorem{remark}[theorem]{Remark}

\numberwithin{equation}{section}
\numberwithin{theorem}{section}

\allowdisplaybreaks

\newcommand{\M}{{\mathbb M}}
\newcommand{\N}{{\mathbb N}}

\newcommand{\R}{{\mathbb {R}}}

\newcommand{\F}{{\mathcal F}}

\renewcommand{\leq}{\leqslant}
\renewcommand{\geq}{\geqslant}

\newcommand{\OO}{{\mathcal O}}

\newcommand{\DM}{{\mathcal D}}

\newcommand{\et}{{\tilde\eta}}
\newcommand{\lxm}{{\lambda\xi\mu}}
\newcommand{\lIxm}{{\lambda_1\xi\mu}}
\newcommand{\lIIxm}{{\lambda_2\xi\mu}}

\newcommand{\id}{\text{id}}

\DeclareMathOperator{\supp}{supp}
\DeclareMathOperator{\tr}{tr}

\newcommand{\at}{\makeatletter @\makeatother}

\allowdisplaybreaks

\begin{document}

\title[Stochastic Functional Differential Equations and Sensitivity to their Initial Path]{Stochastic Functional Differential Equations \\and Sensitivity to their Initial Path}

\date{January 22, 2017 }

\author[Ba\~nos]{D. R. Ba\~nos}
\address{D. R. Ba\~nos: Department of Mathematics, University of Oslo,
PO Box 1053 Blindern, N-0316 Oslo, Norway.Email: davidru\at math.uio.no}

\author[Di Nunno]{G. Di Nunno}
\address{G. Di Nunno: Department of Mathematics, University of Oslo,
PO Box 1053 Blindern, N-0316 Oslo, Norway, and, Norwegian School of Economics and Business Administration, Helleveien 30, N-5045 Bergen, Norway. Email: giulian\at math.uio.no}

\author[Haferkorn]{H. H. Haferkorn}
\address{H. H. Haferkorn: Department of Mathematics, University of Oslo,
PO Box 1053 Blindern, N-0316 Oslo, Norway.Email: hanneshh\at math.uio.no}

\author[Proske]{F. Proske}
\address{F. Proske: Department of Mathematics, University of Oslo,
PO Box 1053 Blindern, N-0316 Oslo, Norway. Email: proske\at math.uio.no}


\maketitle

\begin{abstract}
We consider systems with memory represented by stochastic functional differential equations. Substantially, these are stochastic differential equations with coefficients depending on the past history of the process itself. Such coefficients are hence defined on a functional space. 
Models with memory appear in many applications ranging from biology to finance.
Here we consider the results of some evaluations based on these models (e.g. the prices of some financial products) and the risks connected to the choice of these models. In particular we focus on the impact of the initial condition on the evaluations. 
This problem is known as the analysis of sensitivity to the initial condition and, in the terminology of finance, it is referred to as the Delta. 
In this work the initial condition is represented by the relevant past history of the stochastic functional differential equation. This naturally leads to the redesign of the definition of Delta. We suggest to define it as a functional directional derivative, this is a natural choice. For this we study a representation formula which allows for its computation without requiring that the evaluation functional is differentiable. This feature is particularly relevant for applications.
Our formula is achieved by studying an appropriate relationship between Malliavin derivative and functional directional derivative. 
For this we introduce the technique of {\it randomisation of the initial condition}.
\end{abstract}


\vspace{5mm}
\section{Introduction}

Several phenomena in nature show evidence of both a stochastic behaviour and a dependence on the past history when evaluating the present state. Examples of models taking into account both features come from biology in the different areas of population dynamics, see e.g. \cite{Mao1, Mao2}, or gene expression, see e.g.  \cite{MPBF11}, or epidemiology, see e.g. \cite{BG13}. We find several stochastic models dealing with delay and memory also in the different areas of economics and finance. 
The delayed response in the prices of both commodities and financial assets is studied for example in \cite{AI05, AIK05, SMohammed, SMohammedetal, Chang, Chang2, KSW05, KSW07, Kuchler, Stoica, S13}. The very market inefficiency and also the fact that traders persistently use past prices as a guide to decision making induces memory effects that may be held responsible for market bubbles and crashes. See e.g. \cite{ARS08, HR98}.

In this work we consider a general stochastic dynamic model incorporating delay or memory effects. Indeed we consider stochastic functional differential equations (SFDE), which are substantially stochastic differential equations with coefficients depending on the past history of the dynamic itself.
These SFDEs have already been studied in the pioneering works of \cite{SMohammed2, SMohammed3, Yan} in the Brownian framework. The theory has later been developed including models for jumps in \cite{BCDDR16}.
From another perspective models with memory have been studied via the so-called functional It\^o calculus as introduced in \cite{D09} and then developed steadily in e.g. \cite{CF10, CF13}. For a comparison of the two approaches we refer to e.g. \cite{CR16, FZ16}. In the deterministic framework functional differential equations are widely studied. See, e.g. \cite{HVL93}.

By model risk we generically mean all risks entailed in the choice of a model in view of prediction or forecast. One aspect of model risk management is the study of the sensitivity of a model to the estimates of its parameters. In this paper we are interested in the sensitivity to the initial condition. In the terminology of mathematical finance this is referred to as the Delta. However, in the present setting of SFDEs, the very concept of Delta has to be defined as new, being the initial condition an initial path and not only a single initial point as in the standard stochastic differential equations. It is the first time that the sensitivity to the initial path is tackled, though it appears naturally whenever working in presence of memory effects.

As illustration, let us consider the SFDE:
\begin{align*}
\begin{cases}
dx(t)= f(t,x(t),x_t)dt + g(t,x(t),x_t)dW(t), \ \ t\in [0,T] \\
(x(0),x_0) = \eta
\end{cases}
\end{align*}
where by $x(t)$ we mean the evaluation at time $t$ of the solution process and by $x_t$ we mean the segment of past that is relevant for the evaluation at $t$. 
Let us also consider the evaluation $p(\eta)$ at $t=0$ of some value $\Phi(^\eta x(T), ^\eta x_T)$ at $t=T$ of a functional $\Phi$ of the model. Such evaluation is represented as the expectation:
\begin{align}\label{generalPricingFormula}
 p(\eta) = E \left[ \Phi(^\eta x(T), ^\eta x_T) \right].
\end{align}

\noindent We have marked explicitly the dependence on the initial path $\eta$ by an anticipated superindex.

Evaluations of this type are typical in the pricing of financial derivatives, which are financial contracts with payoff $\Psi$ written on an underlying asset with price dynamics $S$ given by an SFDE of the type above. Indeed in this case the classical non arbitrage pricing rule provides a fair price in the form
$$
 p_{risk-neutral}(\eta) = E_{^\eta Q} \left[ \frac{\Psi(^\eta S(T), ^\eta S_T)}{N(T)} \right] = E \left[ ^\eta Z(T) \frac{\Psi(^\eta S(T), ^\eta S_T)}{N(T)} \right], 
$$
where $^\eta Z(T) = \frac{d {}^\eta Q}{dP}$ is the Radon-Nykodim derivative of the risk-neutral probability measure $^\eta Q$ and $N(T)$ is a chosen num\'eraire used for discounting. We observe that such pricing measure $^\eta Q$ depends on $\eta$ by construction.

Analogously, in the so-called benchmark approach to pricing (see e.g. \cite{Platen}), a non-arbitrage fair price is given in the form
$$
 p_{benchmark}(\eta) = E \left[ \frac{\Psi(^\eta S(T), ^\eta S_T)}{^\eta G(T)} \right], 
$$
where $^\eta G(T)$ is the value of an appropriate benchmark process, used in discounting and guaranteeing that the very $P$ is an appropriate pricing measure.
Here we note that the benchmark depends on the initial path $\eta$ of the underlying price dynamics. Both pricing approaches can be represented as \eqref{generalPricingFormula} and from now on we shall generically call \emph{payoff} the functional $\Phi$, borrowing the terminology from finance. 

Then, in the present notations, the study of the sensitivity to the initial condition consists in the study of some derivative of $p(\eta)$:
$$
 \frac{\partial}{\partial \eta} p(\eta) = \frac{\partial}{\partial \eta}   E \left[ \Phi(^\eta x(T), ^\eta x_T) \right] .
$$
and its possible representations.

In this work we interpret the derivative above as a functional directional derivative and we study formulae for its representations. 
Our approach takes inspiration from the seminal papers \cite{Fournieetal, Fournieetal2}. Here Malliavin calculus is used to obtain a nice formula, where the derivative is itself represented as an expectation of the product of the functional $\Phi$ and some random variable, called Malliavin weight.

We remark immediately that the presence of memory has effects well beyond the expected and the formulae we obtain will not be, unfortunately, so elegant. The representation formulae we finally obtain do not formally present or require the Fr\'echet differentiability of $\Phi$. This is particularly relevant for applications e.g. to pricing. To obtain our formulae we shall study the relationship between functional Fr\'echet derviatives and Malliavin derivatives. However, this relationship has to be carefully constructed. Our technique is based on what we call \emph{the randomisation of the initial path condition}, which is based on the use of an independent Brownian noise to ''shake'' the past. 

The paper is organised as follows. In Section 2 we provide a detailed background of SFDEs. The first part of Section 3 is dedicated to the study of the sensitivity to the initial path condition and the technique of randomisation. We obtain a general representation formula for the sensitivity. Here we see that there is a balance between the generality of the functional $\Phi$ allowed and the regularity on the coefficients of the dynamics of the underlying. The second part of Section 3 presents further 
detailed results in the case of a suitable randomisation choice.
The Appendix contains some technical proof, given with the aim of a self-contained reading.

\vspace{5mm}
\section{Stochastic functional differential equations}\label{section2}

In this section we present a general setup for stochastic functional differential equations (SFDEs). Our framework is inspired by and generalises \cite{SMohammed, SMohammedetal} and \cite{Kuchler}.

\vspace{5mm}
\subsection{The model}

On the complete probability space $(\Omega, \mathcal{F},(\mathcal{F}_t)_{t\in [0,T]},P)$ where the filtration satisfies the usual assumptions and is such that $\mathcal{F}=\mathcal{F}_T$, we consider $W=\{W(t,\omega); \ \omega\in \Omega, t\in[0,T]\}$ an $m$-dimensional standard $(\mathcal{F}_t)_{t\in [0,T]}$-Brownian motion. Here $T\in[0,\infty)$.

We are interested in stochastic processes $x:[-r,T]\times \Omega \rightarrow \R^d$, $r\geq 0$, with finite second order moments and a.s. continuous sample paths. So, one can look at $x$ as a random variable $x:\Omega \rightarrow \mathcal{C}([-r,T],\R^d)$ in $L^2(\Omega, \mathcal{C}([-r,T],\R^d))$. In fact, we can look at $x$ as

$$x:\Omega \rightarrow \mathcal{C}([-r,T],\R^d) \hookrightarrow L^2([-r,T],\R^d) \hookrightarrow \R^d \times L^2([-r,T],\R^d)$$
where the notation $\hookrightarrow$ stands for \emph{continuously embedded in}, which holds since the domains are compact.

From now on, for any $u \in [0,T]$, we write $M_2([-r,u],\R^d):= \R^d \times L^2([-r,u],\R^d)$ for the so-called Delfour-Mitter space endowed with the norm

\begin{align}\label{norm}
\|(v,\theta)\|_{M_2} = \left(|v|^2+ \|\theta\|_2^2\right)^{1/2},\quad(v,\theta) \in M_2([-r,u],\R^d),
\end{align}

\noindent where $\|\cdot\|_2$ stands for the $L^2$-norm and $|\cdot|$ for the Euclidean norm in $\R^d$. For short we denote $M_2 := M_2([-r,0],\R^d)$.

The interest of using such space comes from two facts. On the one hand, the space $M_2$ endowed with the norm \eqref{norm} has a Hilbert structure which allows for a Fourier representation of its elements. On the other hand, as we will see later on, the point 0 plays an important role and therefore we need to distinguish between two processes in $L^2([-r,0], \R^d)$ that have different images at the point 0. In general the spaces $M_2([-r,u], \R^d) $ are also natural to use since they coincide with the corresponding spaces of continuous functions $\mathcal{C}([-r,u],\R^d)$ completed with respect to the norm (\ref{norm}), by taking the natural injection $i(\varphi(\cdot)) = (\varphi(u), \varphi(\cdot)1_{[-r,u)})$ for a $\varphi\in \mathcal{C}([-r,u],\R^d)$ and by closing it.

Furthermore, by the continuous embedding above, we can consider the random process $x:\Omega \times [-r,u] \longrightarrow \R^d$ as a random variable
$$
x: \Omega \longrightarrow M_2([-r,u], \R^d) 
$$
in $L^2(\Omega, M_2([-r,u],\R^d))$, that is
$$
\lVert x \rVert_{L^2(\Omega, M_2([-r,u],\R^d))} = \left(\int_{\Omega} \|x(\omega)\|_{M_2([-r,u],\R^d)}^2 P(d\omega) \right)^{1/2} < \infty.
$$
For later use, we write $L_A^2(\Omega,M_2([-r,u],\R^d))$ for the subspace of $L^2(\Omega,M_2([-r,u],\R^d))$ of elements that admit an $(\mathcal{F}_t)_{t\in [0,u]}$-adapted modification.

\vspace{2mm}
To deal with memory and delay we use the concept of segment of $x$. 
Given a process $x$, some delay gap $r>0$, and a specified time $t\in[0,T]$, the {\it segment of $x$} in the past time interval $[t-r,t]$ is denoted by $x_t(\omega,\cdot): [-r,0] \rightarrow \R^d$ and it is defined as 
$$
x_t(\omega,s):= x(\omega, t+s), \qquad s\in[-r,0].
$$
So $x_t(\omega,\cdot)$ is the segment of the $\omega$-trajectory of the process $x$, and contains all the information of the past down to time $t-r$. 
In particular, the segment of $x_0$ relative to time $t=0$ is the initial path and carries the information about the process from before $t=0$. 

Assume that, for each $\omega\in\Omega$, $x(\cdot,\omega)\in L^2([-r,T], \R^d)$. 
Then $x_t(\omega)$ can be seen as an element of $L^2([-r,0], \R^d)$ for each $\omega\in\Omega$ and $t\in[0,T]$. 
Indeed the couple $(x(t),x_t)$ is a $\mathcal{F}_t$-measurable random variable with values in $M_2$, i.e. $(x(t,\omega),x_t(\omega, \cdot)) \in M_2$, given $\omega\in \Omega$.

\vspace{2mm}
Let us consider an $\F_0$-measurable random variable $\eta\in L^2(\Omega,M_2)$. 
To shorten notation we write $\mathbb{M}_2 := L^2(\Omega,M_2)$.
A stochastic functional differential equation (SFDE), is written as

\begin{align}\label{delayeq0}
\begin{cases}
dx(t)= f(t,x(t),x_t)dt + g(t,x(t),x_t)dW(t), \ \ t\in [0,T] \\
(x(0),x_0)=\eta \in \mathbb{M}_2
\end{cases}
\end{align}
where
\begin{align*}
 f:[0,T]\times M_2\rightarrow \R^d\quad\text{ and }\quad g:[0,T]\times M_2\rightarrow L(\R^m,\R^d).
\end{align*}

\vspace{5mm}
\subsection{Existence and uniqueness of solutions}
Under suitable hypotheses on the functionals $f$ and $g$, one obtains existence and uniqueness of the strong solution (in the sense of $L^2$) of the SFDE (\ref{delayeq0}). The solution is a process $x\in L^2 (\Omega, M_2([-r,T],\R^d))$ admitting an $(\mathcal{F}_t)_{t\in [0,T]}$-adapted modification, that is, $x\in L_A^2(\Omega, M_2([-r,T],\R^d))$. 

We say that two processes $x^1,x^2\in L^2(\Omega, M_2([-r,T],\R^d))$ are $L^2$-unique, or unique in the $L^2$-sense if 
$\|x_1-x_2\|_{ L^2(\Omega, M_2([-r,T],\R^d))}=0$.

\vspace{3mm}
\noindent
\textbf{Hypotheses (EU):}
\begin{enumerate}
\item[(EU1)]  (Local Lipschitzianity) The drift and the diffusion functionals $f$ and $g$ are Lipschitz on bounded sets in the second variable uniformly w.r.t. the first, i.e., for each integer $n\geq 0$, there is a Lipschitz contant $L_n$ independent of $t\in[0,T]$ such that,
$$|f(t,\varphi_1)-f(t,\varphi_2)|_{\R^d}+\|g(t,\varphi_1)-g(t,\varphi_2)\|_{L(\R^m,\R^d)} \leq L_n \|\varphi_1 - \varphi_2\|_{M_2}$$
for all $t\in [0,T]$ and functions $\varphi_1,\varphi_2\in M_2$ such that $\|\varphi_1\|_{M_2}\leq n$, $\|\varphi_2\|_{M_2}\leq n$.
\item[(EU2)] (Linear growths) There exists a constant $C>0$ such that,
$$|f(t,\psi)|_{\R^d}+\|g(t,\psi)\|_{L(\R^m,\R^d)} \leq C \left( 1+ \|\psi\|_{M_2}\right)$$
for all $t\in[0,T]$ and $\psi\in M_2$.
\end{enumerate}

\vspace{2mm}
The following result belongs to \cite[Theorem 2.1]{SMohammed2}. Its proof is based on an approach similar to the one in the classical deterministic case based on successive Picard approximations.
\begin{theorem}[Existence and Uniqueness]\label{ExistenceAndUniquenessSFDE}
Given Hypotheses \textbf{\emph{(EU)}} on the coefficients $f$ and $g$ and the initial condition $\eta\in\M_2$, the SFDE (\ref{delayeq0}) has a solution $^\eta x \in L_A^2(\Omega,M_2([-r,T],\R^d))$ which is unique in the sense of $L^2$. The solution (or better its adapted representative) is a process $^\eta x: \Omega \times [-r,T] \rightarrow \R^d$ such that
\begin{enumerate}
\item[(1)] $^{\eta} x (t) = \eta(t)$, $t\in [-r,0]$.
\item[(2)] $^{\eta} x (\omega) \in M_2([-r,T],\R^d)$ $\omega$-a.s.
\item[(3)] For every $t\in[0,T]$, $^{\eta} x(t) : \Omega \rightarrow \R^d$ is $\mathcal{F}_t$-measurable.
\end{enumerate}
\end{theorem}

\noindent
From the above we see that it makes sense to write
$$
^{\eta} x(t)=
\begin{cases}\eta(0) + \int_0^t f(u, \ ^{\eta} x(u), \ ^{\eta} x_u)du + \int_0^t g(u, \ ^{\eta} x(u), \ ^{\eta} x_u)dW(u), \ t\in[0,T] \\
\eta(t), \ t\in[-r,0].
\end{cases}
$$
Observe that the above integrals are well defined. In fact,  the process 
$$
(\omega,t) \mapsto ( \ ^{\eta} x(t,\omega), \ ^{\eta}  x_t(\omega))
$$ 
belongs to $\M_2$ and is adapted since $x$ is pathcontinuous and adapted and its composition with the \mbox{deterministic} coefficients $f$ and $g$ is then adapted as well. Note that $^{\eta} x$ represents the solution starting off at time $0$ with initial condition $\eta\in \mathbb{M}_2$.

\vspace{2mm}
One could consider the same dynamics but starting off at a later time, let us say, \mbox{$s\in(0,T]$}, with initial condition $\eta\in \mathbb{M}_2$. Namely, we could consider:
\begin{align}\label{delayeq2}
\begin{cases}
dx(t)= f(t,x(t),x_t)dt + g(t,x(t),x_t)dW(t), \quad t\in [s,T] \\
x(t)=\eta(t-s), \quad t\in [s-r,s].
\end{cases}
\end{align}
Again, under \textbf{(EU)} the SFDE (\ref{delayeq2}) has the solution,
\begin{align}\label{delayeq2integral}
^{\eta} x^s (t)=
\begin{cases}\eta(0) + \int_s^t f(u, \ ^{\eta} x^s(u), \ ^{\eta} x_u^s)du + \int_s^t g(u, \ ^{\eta} x^s(u), \ ^{\eta} x_u^s )dW(u), \quad t\in[s,T] \\
\eta(t-s), \quad t\in[s-r,s]
\end{cases}
\end{align}

The right-hand side superindex in $^\eta x^s$ denotes the starting time. 
We will omit the superindex when starting at 0, $^\eta x^0 =  \ ^\eta x$. The interest of defining the solution to \eqref{delayeq2} starting at any time $s$ comes from the semigroup property of the flow of the solution which we present in the next subsection.
For this reason we introduce the notation 
\begin{align}\label{SemiflowNotation}
 X^s_t(\eta, \omega) := X(s,t, \eta, \omega) := (^\eta x^s(t, \omega), ^\eta x^s_t(\omega)), \qquad \omega\in \Omega, \ s \leq t.
\end{align}

In relation to (\ref{delayeq2}) we also define the following evaluation operator: $$\rho_0 : M_2 \rightarrow \R^d,\quad\rho_0\varphi := v\quad\text{for any }\varphi=(v, \theta)\in M_2.$$ We observe here that the random variable $^\eta x^s(t)$ is an evaluation at 0 of the process $X_t^s(\eta)$, $t\in [s,T]$.

\vspace{2mm}
\subsection{Differentiability of the solution}
We recall that our goal is the study of the influence of the initial path $\eta$ on the functionals of the solution of (\ref{delayeq0}).
For this we need to ensure the existence of an at-least-once differentiable stochastic flow for (\ref{delayeq0}).
Hereafter we discuss the differentiability conditions on the coefficients of the dynamics to ensure such property on the flow.

In general, suppose we have $E$ and $F$ Banach spaces, $U\subseteq E$ an open set and $k\in\N$. We write $L^k(E,F)$ for the space of continuous $k$-multilinear operators $A:E^k\rightarrow F$ endowed with the uniform norm $$\lVert A\rVert_{L^k(E,F)}:=\sup\{\lVert A(v_1,\dots,v_k)\rVert_F,\,\lVert v_i\rVert_E\leq1,\,i=1,\dots,k\}.$$
Then an operator $f:U\rightarrow F$ is said to be of class $\mathcal{C}^{k,\delta}$ if it is $C^k$ and $D^k f:U \rightarrow L^k(E,F)$ is $\delta$-H\"older continuous on bounded sets in $U$. Moreover, $f:U\rightarrow F$ is said to be of class $\mathcal{C}_b^{k,\delta}$ if it is $C^k$, $D^k f:U \rightarrow L^k(E,F)$ is $\delta$-H\"older continuous on $U$, and all its derivatives $D^j f$, $1\leq j\leq k$ are globally bounded on $U$. The derivative $D$ is taken in the Fr\'echet sense. 

\vspace{2mm}
First of all we consider SFDEs in the special case when
$$
g(t,(\varphi(0),\varphi(\cdot)))=g(t,\varphi(0)), \quad \varphi=(\varphi(0),\varphi(\cdot))\in \mathbb{M}_2
$$
that is, $g$ is actually a function $[0,T]\times\R^d\rightarrow\R^{d\times m}$.\

For completeness we give the definition of stochastic flow.


\begin{definition}\label{stochflow}
Denote by $S([0,T]):=\{s,t\in [0,T]: 0\leq s<t<T\}$. Let $E$ be a Banach space. A stochastic $\mathcal{C}^{k,\delta}$-semiflow on $E$ is a measurable mapping $X: S([0,T]) \times E\times \Omega \rightarrow E$ satisfying the following properties:
\begin{enumerate}
\item[(i)] For each $\omega \in \Omega$, the map $X(\cdot,\cdot,\cdot,\omega): S([0,T]) \times E\rightarrow E$
is continuous.
\item[(ii)] For fixed $(s,t,\omega) \in S([0,T])\times\Omega$ the map $X(s,t,\cdot,\omega): E  \rightarrow E$
is $\mathcal{C}^{k,\delta}$.
\item[(iii)] For $0\leq s\leq u\leq t$, $\omega\in \Omega$ and $x\in E$, the property
$X(s,t,\eta,\omega)=X(u,t,X(s,u,\eta,\omega),\omega)$ holds.
\item[(iv)] For all $(t,\eta,\omega)\in [0,T]\times E \times \Omega$, one has $X(t,t,\eta,\omega)=\eta$.
\end{enumerate}
\end{definition}

\noindent In our setup, we consider the space $E=M_2$.\\
\par
\noindent\textbf{Hypotheses (FlowS):}

\begin{enumerate}
\item[(FlowS1)] The function $f:[0,T]\times M_2\rightarrow \R^d$ is jointly continuous; the map $M_2\ni\varphi\mapsto f(t,\varphi)$ is Lipschitz on bounded sets in $M_2$ and $\mathcal{C}^{1,\delta}$ uniformly in $t$ (i.e. the $\delta$-H\"older constant is uniformly bounded in $t\in[0,T]$) for some $\delta\in(0,1]$.
\item[(FlowS2)] The function $g:[0,T]\times \R^d\rightarrow \R^{d\times m}$ is jointly continuous; the map $\R^d\ni v\mapsto g(t,v)$ is  $\mathcal{C}_b^{2,\delta}$ uniformly in $t$.
\item[(FlowS3)] \textbf{One} of the following conditions is satisfied:
\begin{enumerate}
\item[(a)]There exist $C>0$ and $\gamma\in[0,1)$ such that
$$|f(t,\varphi)|\leq C(1+\lVert\varphi\rVert_{M_2}^\gamma)$$
for all $t\in[0,T]$ and all $\varphi\in M_2$
\item[(b)] For all $t\in[0,T]$ and $\varphi\in M_2$, one has $f(t,\varphi,\omega)=f(t,\varphi(0),\omega)$. Moreover, it exists $r_0\in(0,r)$ such that $$f(t,\varphi,\omega)=f(t,\tilde\varphi,\omega)$$ for all $t\in[0,T]$ and all $\tilde\varphi$ such that $\varphi(\cdot)1_{[-r,-r_0]}(\cdot)=\tilde\varphi(\cdot)1_{[-r,-r_0]}(\cdot)$.
\item[(c)] For all $\omega\in\Omega$,
$$\sup_{t\in[0,T]}\lVert(D\psi(t,v,\omega))^{-1}\rVert_{M_2}<\infty,$$
where $\psi(t,v)$ is defined by the stochastic differential equation
\begin{align*}
\begin{cases}
 d\psi(t,v)=g(t,\psi(t,v))dW(t),\\
 \psi(0,v)=v.
\end{cases}
\end{align*}
Moreover, there exists a constant C such that
$$|f(t,\varphi)|\leq C(1+\lVert\varphi\rVert_{M_2})$$
for all $t\in[0,T]$ and $\varphi\in M_2$.
\end{enumerate}
\end{enumerate}

\noindent Then, \cite[Theorem 3.1]{SMohammed3} states the following theorem.

\begin{theorem}\label{SemiflowSpecial}
 Under Hypotheses \textbf{\emph{(EU)}} and \textbf{\emph{(FlowS)}}, $X^s_t(\eta,\omega)$ defined in \eqref{SemiflowNotation} is a $\mathcal{C}^{1,\varepsilon}$-semiflow for every $\varepsilon\in(0,\delta)$.
\end{theorem}

Next, we can consider a more general diffusion coefficient $g$ following the approach introduced in \cite[Section 5]{SMohammed3}. Let us assume that the function $g$ is of type:

\begin{align*}
 g(t,(x(t),x_t))=\bar g(t,x(t),a+\int_0^th(s,(x(s),x_s))ds),
\end{align*}
for some constant $a$ and some functions $\bar g$ and $h$ satisfying some regularity conditions that will be specified later. This case can be transformed into a system of the previous type where the diffusion coefficient does not explicitly depend on the segment. In fact, defining $y(t):=(y^{(1)}(t),y^{(2)}(t))^\top$ where $y^{(1)}(t):=x(t)$, $t\in[-r,T]$, $y^{(2)}(t):=a+\int_0^th(s,(x(s),x_s))ds$, $t\in[0,T]$ and $y^{(2)}(t):=0$ on $[-r,0]$, we have the following dynamics for $y$:
\begin{align}\label{transformeddelayeq}
\begin{cases}
 dy(t)=F(t,y(t),y_t)dt+G(t,y(t))dW(t),\\
 y(0)=(\eta(0),a)^\top,\, y_0=(\eta,0)^\top,
\end{cases}
\end{align}

\noindent where
\begin{align}
 F(t,y(t),y_t)=\begin{pmatrix}
                f(t,y^{(1)}(t),y^{(1)}_t)\\
                h(t,y^{(1)}(t),y^{(1)}_t)
               \end{pmatrix},\,
 G(t,y(t))=\begin{pmatrix}
            \bar g(t,y^{(1)}(t),y^{(2)}(t))\\
            0
           \end{pmatrix}.
\end{align}

\noindent The transformed system \eqref{transformeddelayeq} is now an SFDE of type \eqref{delayeq0} where the diffusion coefficient does not explicitely depend on the segment. That is the differentiability of the flow can be studied under the corresponding Hypotheses \textbf{(FlowS)}. Hereafter, we specify the conditions on $\bar g$ and $h$ so that Hypotheses \textbf{(EU)} and \textbf{(FlowS)} are satisfied by the transformed system \eqref{transformeddelayeq}. Since the conditions \textbf{(FlowS3)(a)} and \textbf{(b)} are both too restrictive for \eqref{transformeddelayeq}, we will make sure that \textbf{(FlowS3)(c)} is satisfied. Under these conditions we can guarantee the differentiability of the solutions to the SFDE \eqref{delayeq2} for the above class of diffusion coefficient $g$.\\
\par
\noindent\textbf{Hypotheses (Flow):}

\begin{enumerate}
\item[(Flow1)] $f$ satisfies \textbf{(FlowS1)} and there exists a constant C such that
$$|f(t,\varphi)|\leq C(1+\lVert\varphi\rVert_{M_2})$$
for all $t\in[0,T]$ and $\varphi\in M_2$.
\item[(Flow2)] $g(t,\varphi)$ is of the following form
\begin{align*}
 g(t,\varphi)=\bar g(t,v,\tilde g(\theta)),\quad t\in[0,T],\quad \varphi=(v,\theta)\in M_2
\end{align*}
where $\bar g$ satisfies the following conditions:
\begin{enumerate}
 \item[(a)] The function $\bar g:[0,T]\times \R^{d+k}\rightarrow \R^{d\times m}$ is jointly continuous; the map $\R^{d+k}\ni y\mapsto \bar g(t,y)$ is  $\mathcal{C}_b^{2,\delta}$ uniformly in $t$.
 \item[(b)] For each $v\in\R^{d+k}$, let $\{\Psi(t,v)\}_{t\in[0,T]}$ solve the stochastic differential equation
 \begin{align*}
  \Psi(t,v)=v+\begin{pmatrix}
              \int_0^t\bar g(s,\Psi(s,v))dW(s)\\
              0
             \end{pmatrix},
 \end{align*}
 where $0$ denotes the null-vector in $\R^k$. Then $\Psi(t,v)$ is Fr\'echet differentiable w.r.t. $v$ and the Jacobi-matrix $D\Psi(t,v)$ is invertible and fulfils, for all $\omega\in\Omega$,
 \begin{align*}
  \sup_{\stackrel{t\in[0,T]}{v\in\R^{d+k}}}\|D\Psi^{-1}(t,v,\omega)\|<\infty,\,\text{where }\|\cdot\|\text{ denotes any matrix norm.}
 \end{align*}
\end{enumerate} 
and, $\tilde g: L^2([-r,0],\R^d)\rightarrow\R^k$ satisfies the following conditions:
\begin{enumerate}
 \item[(c)] It exists a jointly continuous function $h:[0,T]\times M_2\rightarrow\R^k$ s.t. for each $\tilde\varphi \in L^2([-r,T],\R^d)$,
 \begin{align*}
 \tilde g(\tilde\varphi_t)=\tilde g(\tilde\varphi_0)+\int_0^th(s,(\tilde\varphi(s),\tilde\varphi_s))ds,
 \end{align*}
 where $\tilde\varphi_t\in L^2([-r,0],\R^d)$ is the segment at $t$ of a representative of $\tilde\varphi$.
 \item[(d)] $M_2\ni\varphi\mapsto h(t,\varphi)$ is Lipschitz on bounded sets in $M_2$, uniformly w.r.t. $t\in[0,T]$ and $\mathcal{C}^{1,\delta}$ uniformly in $t$.
\end{enumerate}
\end{enumerate}

\begin{corollary}
 Under Hypotheses \textbf{\emph{(Flow)}}, the solution $X_t^s(\eta)=X(s,t,\eta,\omega)$, $\omega\in \Omega$, $t\geq s$ to \eqref{delayeq2} is a $\mathcal{C}^{1,\varepsilon}$-semiflow for every $\varepsilon\in(0,\delta)$. In particular, $\varphi\mapsto X(s,t,\varphi,\omega)$ is $C^1$ in the Fr\'echet sense.
\end{corollary}

\vspace{2mm}
\section{Sensitivity analysis to the initial path condition}\label{section3}

From now on, we consider a stochastic process $x$ which satisfies dynamics \eqref{delayeq0}, where the coefficients $f$ and $g$ are such that conditions \textbf{(EU)} and \textbf{(Flow)} are satisfied.

\vspace{2mm}
Our final goal is to study the sensitivity of evaluations of type

\begin{align}\label{pricedef}
p(\eta)=E\left[ \Phi(X^0_T(\eta))\right]=E\left[ \Phi( ^{\eta} x(T), \ ^{\eta} x_T)\right], \ \eta \in \M_2
\end{align}

\noindent to the initial path in the model $^{\eta} x$. Here, $\Phi: M_2 \rightarrow \R$ is such that $\Phi(X_T^0(\eta))\in L^2(\Omega,\R)$. The sensitivity will be interpreted as the directional derivative

\begin{align}\label{DirectionalDerivative}
\partial_hp(\eta):=\frac{d}{d\varepsilon} p(\eta + \varepsilon h) \bigg|_{\varepsilon=0}= \lim_{\varepsilon \to 0} \frac{p(\eta+\varepsilon h) - p(\eta)}{\varepsilon}, \ \ h\in M_2.
\end{align}

\noindent Hence we shall study pertubations direction $h\in M_2$. The final aim is to give a representation of $\partial_hp(\eta)$ in which the function $\Phi$ is not directly differentiated. This is in the line  with the representation of the sensitivity parameter Delta by means of weights. See, e.g. the Malliavin weight introduced in \cite{Fournieetal, Fournieetal2} for the classical case of no memory. For this we impose some stronger regularity conditions on $f$ and $g$:\\
\par
\noindent\textbf{Hypotheses (H):}

\begin{enumerate}
\item[(H1)] (Global Lipschitzianity) $\varphi\mapsto f(t,\varphi)$, $\varphi\mapsto g(t,\varphi)$ globally Lipschitz uniformly in $t$ with Lipschitz constants $L_f$ and $L_g$, i.e.
\begin{align*}
 |f(t,\varphi_1)-f(t,\varphi_2)|_{\R^d}\leq L_f \|\varphi_1 - \varphi_2\|_{M_2}\\
 \|g(t,\varphi_1)-g(t,\varphi_2)\|_{L(\R^m,\R^d)} \leq L_g \|\varphi_1 - \varphi_2\|_{M_2}
\end{align*}
for all $t\in[0,T]$ and $\varphi_1, \varphi_2\in M_2$.
\item[(H2)] (Lipschitzianity of the Fr\'echet derivatives) $\varphi\mapsto Df(t,\varphi)$, $\varphi\mapsto Dg(t,\varphi)$ are globally Lipschitz uniformly in $t$ with Lipschitz constants $L_{Df}$ and $L_{Dg}$, i.e.
\begin{align*}
 \|Df(t,\varphi_1)-Df(t,\varphi_2)\|\leq L_{Df} \|\varphi_1 - \varphi_2\|_{M_2}\\
 \|Dg(t,\varphi_1)-Dg(t,\varphi_2)\|\leq L_{Dg} \|\varphi_1 - \varphi_2\|_{M_2}
\end{align*}
for all $t\in[0,T]$ and $\varphi_1, \varphi_2\in M_2$.
\end{enumerate}

\noindent The corresponding stochastic $\mathcal{C}^{1,1}$-semiflow is again denoted by $X$.

\vspace{2mm}
Before proceeding, we give a simple example of SFDE satisfying all assumptions \textbf{(EU)}, \textbf{(Flow)} and \textbf{(H)}.

\begin{example}\label{LinearCoefficientsExample}
 Consider the SFDE \eqref{delayeq0} where the functions $f$ and $g$ are given by
 \begin{align*}
  f(t,\varphi)=M(t)\varphi(0)+\int_{-r}^0\bar M(s)\varphi(s)ds,\\
  g(t,\varphi)=\Sigma(t)\varphi(0)+\int_{-r}^0\bar\Sigma(s)\varphi(s)ds,
 \end{align*}
 where $M:[0,T]\rightarrow\R^{d\times d}$, $\bar M:[-r,0]\rightarrow\R^{d\times d}$, $\Sigma:[0,T]\rightarrow L(\R^d,\R^{d\times m})$, and $\bar\Sigma:[-r,0]\rightarrow L(\R^d,\R^{d\times m})$ are bounded differentiable functions, $\bar\Sigma(-r)=0$ and $s\mapsto \bar \Sigma'(s)=\frac{d}{ds}\bar \Sigma(s)$ are bounded as well.

 \vspace{2mm}
Obviously, $f$ and $g$ satisfy \textbf{(EU)} and \textbf{(H)} and therefore also \textbf{(Flow1)}. In order to check conditions \textbf{(Flow2)}, we note that
 \begin{align*}
  g(t,\varphi)&=\bar g(t,\varphi(0),\tilde g(\varphi(\cdot))),
 \end{align*}
 where
 \begin{align*}
  \bar g(t,y)=\Sigma(t)y^{(1)}+y^{(2)},\,y=(y^{(1)},y^{(2)})^\top,\text{ and }\,\tilde g(\varphi(\cdot))=\int_{-r}^0\bar\Sigma(s)\varphi(s)ds.
 \end{align*}
 The function $\bar g$ satisfies condition \textbf{(Flow2)(a)} as $\Sigma$ is bounded and continuous. Let us check condition \textbf{\emph{(Flow2)(b)}} in the case $d=m=1$. Then $\bar g(t,y)=\sigma(t)y^{(1)}+y^{(2)}$, where $\sigma$ is a real valued, differentiable function and $\Psi$ fulfils the two-dimensional stochastic differential equation
 \begin{align*}
 \begin{cases}
  \Psi^{(1)}(t,v)=v^{(1)}+\int_0^t\bar \sigma(s)\Psi^{(1)}(s,v)+v^{(2)}dW(s),\\
  \Psi^{(2)}(t,v)=v^{(2)},
 \end{cases}
 \end{align*}
 which has the solution
 \begin{align*}
  \Psi^{(1)}(t,v)&=\tilde\Psi(t)\left(v^{(1)}-\int_0^t \sigma(s)v^{(2)}\tilde\Psi^{-1}(s)ds+\int_0^tv^{(2)}\tilde\Psi^{-1}(s)dW(s)\right),\quad \Psi^{(2)}(t,v)=v^{(2)},
 \end{align*}
  with
 \begin{align*} 
  \tilde\Psi(t)&=\exp\left\{-\int_0^t\sigma^2(s)ds+\int_0^t\sigma(s)dW(s)\right\}.
 \end{align*}
 Therefore, we get that
 \begin{align*}
  D\Psi(t,v)&=\begin{pmatrix}
              1+\tilde\Psi(t) & \tilde\Psi(t)\left(-\int_0^t\sigma(s)\tilde\Psi^{-1}(s)ds+\int_0^t\tilde\Psi^{-1}(s)dW(s)\right)\\
              0 & 1
             \end{pmatrix}
 \end{align*}
 and
 \begin{align*}
  D\Psi^{-1}(t,v)&=\begin{pmatrix}
              \frac{1}{1+\tilde\Psi(t)} & -\frac{\tilde\Psi(t)}{1+\tilde\Psi(t)}\left(-\int_0^t\sigma(s)\tilde\Psi^{-1}(s)ds+\int_0^t\tilde\Psi^{-1}(s)dW(s)\right)\\
              0 & 1
             \end{pmatrix}    
 \end{align*}
 Using in fact that $\tilde\Psi(t)>0$ and applying the Frobenius norm $\|\cdot\|_F$, we obtain $\omega$-a.e.
 \begin{align*}
  \|D\Psi^{-1}(t,v)\|_F&=\tr\left((D\Psi^{-1}(t,v))^\top D\Psi^{-1}(t,v)\right)\\
  &\leq 2+\tilde\Psi^2(t)\left(-\int_0^t\sigma(s)\tilde\Psi^{-1}(s)ds+\int_0^t\tilde\Psi^{-1}(s)dW(s)\right)^2<\infty,
 \end{align*}
for $t\in[0,T]$, $v\in\R^2$. By this Hypothesis \textbf{(Flow2)(b)} is fulfilled.\\
 \par
 
Moreover, a simple application of partial integration and Fubini's theorem together with the fact that $\bar\Sigma(-r)=0$ shows that
 \begin{align*}
  \tilde g(\tilde \varphi_t)&=\int_{-r}^0\bar\Sigma(s)\tilde\varphi_t(s)ds=\int_{-r}^0\bar\Sigma(s)\tilde\varphi_0(s)ds+\int_0^t\bigg\{\Sigma(0)\tilde\varphi(u)-\int_{-r}^0\bar\Sigma'(s)\tilde\varphi_u(s)ds\bigg\}du\\
  &=\tilde g(\tilde \varphi_0)+\int_0^th(t,\tilde\varphi(u),\tilde\varphi_u)du.
 \end{align*}
 It can be easlily checked that $h(t,\varphi)=\Sigma(0)\varphi(0)-\int_{-r}^0\bar\Sigma'(s)\varphi(s)ds$ satisfies the conditions given in \textbf{\emph{(Flow2)(c)}} and \textbf{\emph{(d)}}. 
\end{example}

We are now ready to introduce two technical lemmas needed to prove our main results.

\begin{lemma}\label{DiffofSegmentequalsSegmentofDiff}
 Assume that the solution to \eqref{delayeq2} exists and has a $C^{1,1}$-semiflow $X^s_t(\eta,\omega)$, $s\leq t$, $\omega\in\Omega$. Then, the following equality holds for all $\omega\in\Omega$ and all directions $h\in M_2$:
 \begin{align*}
  DX^s_t(\eta,\omega)[h]=(D\,^\eta x^s(t,\omega)[h],D\,^\eta x^s(t+\cdot,\omega)[h])\in M_2.
 \end{align*}
\end{lemma}

\begin{proof}
 Note that $DX^s_t(\eta,\omega)[h]\in M_2$. Let $\{e_i\}_{i\in\N}$ be an orthonormal basis of $M_2$. Then,
 \begin{align*}
  DX^s_t(\eta,\omega)[h]&=\sum_{i=0}^\infty\langle DX^s_t(\eta,\omega)[h],e_i\rangle_{_{M_2}}e_i=\sum_{i=0}^\infty D\langle X^s_t(\eta,\omega),e_i\rangle_{_{M_2}}[h]e_i\\
  &=\sum_{i=0}^\infty D\bigg(x^s(t,\omega)e_i(0)+\int_{-r}^0x^s(t+u,\omega)e_i(u)du\bigg)[h]e_i\\
  &=\sum_{i=0}^\infty \bigg(Dx^s(t,\omega)[h]e_i(0)+\int_{-r}^0Dx^s(t+u,\omega)[h]e_i(u)du\bigg)e_i\\
  &=\sum_{i=0}^\infty \langle (D\,^\eta x^s(t,\omega)[h],D\,^\eta x^s(t+\cdot,\omega)[h]),e_i\rangle_{_{M_2}}e_i\\
  &=(D\,^\eta x^s(t,\omega)[h],D\,^\eta x^s(t+\cdot,\omega)[h]).
 \end{align*}
This finishes the proof.
\end{proof}

\begin{lemma}\label{4thMoments}
 Let Hypotheses \textbf{\emph{(EU)}}, \textbf{\emph{(Flow)}} and \textbf{\emph{(H)}} be fulfilled. Then, for all $t\in[0,T]$, we have that $E[\lVert X^0_t(\eta)\rVert_{M_2}^4]<\infty$ and $E[\lVert DX^0_t(\eta)[h]\rVert_{M_2}^4]<\infty$ and the functions $t\mapsto E[\lVert X^0_t(\eta)\rVert_{M_2}^4]$ and $t\mapsto E[\lVert DX^0_t(\eta)[h]\rVert_{M_2}^4]$ are Lebesgue integrable, i.e.
 \begin{align}
  \int_0^TE[\lVert X^0_t(\eta)\rVert_{M_2}^4]dt&<\infty,\label{Integrability4thMoment}\\
  \int_0^TE[\lVert DX^0_t(\eta)[h]\rVert_{M_2}^4]dt&<\infty\label{Integrability4thMoment2}.
 \end{align}
\end{lemma}

\begin{proof}
 To see this, observe that
 \begin{align*}
  \lVert X^0_s(\eta)\rVert_{M_2}^4&=\Big(|x(s)|^2+\int_{-r}^0 1_{(-\infty,0)}(s+u)|\eta(s+u)|^2du+\int_{-r}^0 1_{[0,\infty)}(s+u)|x(s+u)|^2du\Big)^2\\
  &\leq 3\sup_{t\in[0,T]}|x(t)|^4+3\lVert\eta\rVert_{M_2}^4+3r^2\sup_{t\in[0,T]}|x(t)|^4,
 \end{align*}
 and thus, for all $s\in[0,T]$
 \begin{align}
  E[\lVert X^0_s(\eta)\rVert_{M_2}^4]\leq 3\lVert\eta\rVert_{M_2}^4+3(1+r^2)E[\sup_{t\in[0,T]}|x(t)|^4],\label{4thMomentsHilfsab1}
 \end{align}
 and
 \begin{align}
  \int_0^TE[\lVert X^0_t(\eta)\rVert_{M_2}^4]dt\leq 3T\lVert\eta\rVert_{M_2}^4+3(1+r^2)TE[\sup_{t\in[0,T]}|x(t)|^4]\label{4thMomentsHilfsab2}.
 \end{align}
 To prove \eqref{Integrability4thMoment} it is then enough to show $E[\sup_{t\in[0,T]}|x(t)|^4]<\infty$. Therefore, consider first
 \begin{align*}
  &E[\sup_{t\in[0,T]}|x(t)|^4]\\
  &\quad=E\bigg[\sup_{t\in[0,T]}\Big|\eta(0)+\int_0^tf(s,X^0_s(\eta))ds+\int_0^tg(s,X^0_s(\eta))dW(s)\Big|^4\bigg]\\
  &\quad\leq E\bigg[\sup_{t\in[0,T]}\Big(3\lVert\eta\rVert_{M_2}^2+3\Big(\int_0^tf(s,X^0_s(\eta))ds\Big)^2+3\Big(\int_0^tg(s,X^0_s(\eta))dW(s)\Big)^2\Big)^2\bigg]\\
  &\quad\leq 27\lVert\eta\rVert_{M_2}^4+27T\int_0^TE[|f(s,X^0_s(\eta))|^4]ds+27K_{BDG}E\bigg[\Big(\int_0^T|g(s,X^0_s(\eta))|^2ds\Big)^{\frac{4}{2}}\bigg].
 \end{align*}
 Here we applied twice the fact that $(\sum_{i=1}^na_i)^2\leq n\sum_{i=1}^n|a_i|^2$ as well as Jensen's inequality, Fubini's theorem. Since the process $\int_0^\cdot g(s,X^0_s(\eta))dW(s)$ is a martingale (as a consequence of Theorem \ref{ExistenceAndUniquenessSFDE}), we have also used the Burkholder-Davis-Gundy inequality (with the constant $K_{BDG}$).

 By the linear growth condition \textbf{(EU2)} on $f$ and $g$ and \eqref{4thMomentsHilfsab1}, we have
 \begin{align*}
  |f(s,X^0_s(\eta))|^4&\leq(C(1+\lVert X^0_s(\eta)\rVert_{M_2}))^4\leq8C^4+8C^4\lVert X^0_s(\eta)\rVert_{M_2}^4\\
  &\leq8C^4+24C^4\lVert \eta\rVert_{M_2}^4+24(1+r^2)\sup_{t\in[0,T]}|x(t)|^4,
 \end{align*}
 and the same applies to $|g(s,X^0_s(\eta))|^4$.
 Plugging this in the above estimates, we obtain
 \begin{align*}
  E[\sup_{t\in[0,T]}|x(t)|^4]&\leq 27\lVert\eta\rVert_{M_2}^4(1+24C^4T^2(1+K_{BDG}))+216C^4T^2(1+K_{BDG})\\
  &\quad+648(1+r^2)C^4T^2(1+K_{BDG})E[\sup_{t\in[0,T]}|x(t)|^4],
 \end{align*}
 which is
 \begin{align*}
  (1-T^2k_1^2)E[\sup_{t\in[0,T]}|x(t)|^4]\leq k_2,
 \end{align*}
 where
 \begin{align*}
  k_1&:=\sqrt{648(1+r^2)C^4(1+K_{BDG})}\text{ and}\\
  k_2&:=27\lVert\eta\rVert_{M_2}^4(1+24C^4T^2(1+K_{BDG}))+216C^4T^2(1+K_{BDG}).
 \end{align*}
 Then we distinguish two cases.
 
 \noindent\textbf{Case 1:} $T<\frac{1}{k_1}$. Then $E[\sup_{t\in[0,T]}|x(t)|^4]\leq\frac{k_2}{(1-T^2k_1^2)}$
 Hence, by \eqref{4thMomentsHilfsab1} and \eqref{4thMomentsHilfsab2} we have that \eqref{Integrability4thMoment} holds.

 \noindent\textbf{Case 2:} $T\geq\frac{1}{k_1}$. In this case, choose $0<T_1<T_2<\dots<T_n=T$ for some finite $n$ such that
 \begin{align*}
  T_1&<\frac{1}{k_1}\text{ and }\,T_i-T_{i-1}<\frac{1}{k_1},\quad i=2,\dots,n.
 \end{align*}
 By the semiflow property, we have $X^{T_1}_{T_2}(X^0_{T_1}(\eta))=X^0_{T_2}(\eta)$, so we can solve the SFDE on $[0,T_1]$, and by Case 1 we have
 \begin{align*}
  E[\sup_{t\in[0,T_1]}|x(t)|^4]<\infty\text{ and }\int_0^{T_1}E[\lVert X^0_t(\eta)\rVert_{M_2}^4]dt<\infty.
 \end{align*}
 Then, we use $X^0_{T_1}(\eta)$ as a new starting value and solve the equation on $[T_1,T_2]$. By the same steps as before, we obtain
 \begin{align*}
  E[\sup_{t\in[T_1,T_2]}|x(t)|^4]&\leq\frac{27E[\lVert X^0_{T_1}(\eta)\rVert_{M_2}^4](1+24(T_2-T_1)^2(1+K_{BDG})C^4)+216C^4(T_2-T_1)^2(1+K_{BDG})}{1-648(1+r^2)(T_2-T_1)^2(1+K_{BDG})C^4}<\infty,
 \end{align*}
 and therefore,
 \begin{align*}
  \int_0^{T_2}E[\lVert X^0_t(\eta)\rVert_{M_2}^4]dt&=\int_0^{T_1}E[\lVert X^0_t(\eta)\rVert_{M_2}^4]dt+\int_{T_1}^{T_2}E[\lVert X^0_t(\eta)\rVert_{M_2}^4]dt\\
  &\leq\int_0^{T_1}E[\lVert X^0_t(\eta)\rVert_{M_2}^4]dt+3(T_2-T_1)E[\lVert X^0_{T_1}(\eta)\rVert_{M_2}^4]\\
  &\quad+3(T_2-T_1)(1+r^2)E[\sup_{t\in[T_1,T_2]}|x(t)|^4]<\infty.
 \end{align*}
 Iterating the argument, we conclude that for all $T\in(0,\infty)$, $E[\sup_{t\in[0,T]}|x(t)|^4]<\infty$ and $\int_0^TE[\lVert X^0_t(\eta)\rVert_{M_2}^4]dt<\infty$, that is \eqref{Integrability4thMoment} holds.\\
 \par
 In order to prove \eqref{Integrability4thMoment2}, we define the process
 \begin{align*}
  y(t):=\begin{pmatrix}
         x(t)\\
         Dx(t)[h]
        \end{pmatrix},\,t\in[-r,T]
 \end{align*}
 and the corresponding short-hand notation
 \begin{align*}
  \mathcal{Y}(t,\eta,h)=(X^0_t(\eta),DX^0_t(\eta)[h])\in\M_2\times\M_2
 \end{align*} 
 The process $y$ satisfies the SFDE
 \begin{align}\label{CombinedProcess}
  y(t)&=\begin{pmatrix}
         \eta(0)\\
         h(0)
        \end{pmatrix}
        +\int_0^t\hat f(s,\mathcal{Y}(s,\eta,h))ds+\int_0^t\hat g(s,\mathcal{Y}(s,\eta,h))dW(s),\quad y_0=(\eta,h)
 \end{align}
 where, for $(\varphi,\psi)^\top\in M_2\times M_2$,
 \begin{align*}
  \hat f(s,(\varphi,\psi)):=\begin{pmatrix}
         f(s,\varphi)\\
         Df(s,\varphi)[\psi]
        \end{pmatrix},\quad
        \hat g(s,(\varphi,\psi)):=\begin{pmatrix}
         g(s,\varphi)\\
         Dg(s,\varphi)[\psi]
        \end{pmatrix}.
 \end{align*}

\noindent Thanks to Lemma \ref{DiffofSegmentequalsSegmentofDiff}, we recognize equation \eqref{CombinedProcess} as being of type \eqref{delayeq0}. In fact, we can identify the $M_2\times M_2$-valued random variable $(X^0_s(\eta),DX^0_s(\eta)[h])$ with the $M_2([-r,0],\R^{2d})$-valued random variable $(y(s),y(s+\cdot))$. Using \textbf{(H)} it is now easy to check that $\hat f$ and $\hat g$ fulfil Hypothesis \textbf{(EU)}, which are sufficient for the existence and uniqueness of a solution.

We can therefore argue exactly as in the proof of \eqref{Integrability4thMoment} and obtain that
\begin{align*}
 E[\lVert \mathcal{Y}(t,\eta,h)\rVert_{M_2\times M_2}^4]<\infty\,\forall t\in[0,T]\text{ and }\int_0^TE[\lVert \mathcal{Y}(t,\eta,h)\rVert_{M_2\times M_2}^4]dt<\infty.
\end{align*}
Moreover, since
\begin{align*}
 \lVert \mathcal{Y}(t,\eta,h)\rVert_{M_2\times M_2}^4&=\bigg(|y(t)|_{\R^{2d}}^2+\int_{-r}^0|y(t+u)|_{\R^{2d}}^2\bigg)^2\\
 &=\bigg(|x(t)|_{\R^{d}}^2+|Dx(t)[h]|_{\R^{d}}^2+\int_{-r}^0|x(t+u)|_{\R^{d}}^2+|Dx(t+u)[h]|_{\R^{d}}^2\bigg)^2\\
 &=\big(\lVert X^0_t(\eta)\rVert_{M_2}^2+\lVert DX^0_t(\eta)[h]\rVert_{M_2}^2\big)^2\geq\lVert DX^0_t(\eta)[h]\rVert_{M_2}^4,
\end{align*}
we conclude that $E[\lVert DX^0_t(\eta)[h]\rVert_{M_2}^4]<\infty$ for all $t\in[0,T]$ and \eqref{Integrability4thMoment2} holds.
\end{proof}

Our aim in the study of \eqref{DirectionalDerivative} is to give a formula for $\partial_hp(\eta)$ that avoids differentiating the function $\Phi$. Our approach consists in randomizing the initial condition $\eta$ and in finding a relationship between the Fr\'echet derivative $DX^0_T(\eta)$ applied to a direction $h\in \M_2$ and the Malliavin derivative of the $X^0_T$ with the randomized starting condition.

\subsection{Randomization of the initial condition and the Malliavin derivative}

Following the approaches in, e.g. \cite{Nualart} or \cite{Pronk}, we define an isonormal Gaussian process $\mathbb{B}$ on $L^2([-r,0],\R)$, independent of the $m$-dimensional Wiener process $W$ that drives the SFDE \eqref{delayeq0}. Without loss of generality, we can assume that $W$ and $\mathbb{B}$ are defined on indepentent probability spaces $(\Omega^W,\F^W,P^W)$ and $(\Omega^{\mathbb{B}},\F^{\mathbb{B}}, P^{\mathbb{B}})$ and that $(\Omega,\F,P)=(\Omega^W\times\Omega^{\mathbb{B}},\F^W\otimes\F^{\mathbb{B}}, P^W\otimes P^{\mathbb{B}})$.
From now on we shall work under $\Omega=\Omega^W\times\Omega^{\mathbb{B}}$. Hence, we correspondingly transfer the notation introduced so far to this case. However, we shall deal with the Malliavin and Skorohod calculus only w.r.t. $\mathbb{B}$. In fact, for the isonormal Gaussian process $\mathbb{B}$ we define the Malliavin derivative operator $\DM$ and the Skorohod integral operator $\delta$ as performed in e.g. \cite{Nualart} or \cite{Pronk}.
\par
For immediate use, we give the link between the Malliavin derivative of a segment and the segment of Malliavin derivatives.
\begin{lemma}\label{MalliavinofSegmentequalsSegmentofMalliavin}
 If $X^0_t(\eta)=(\,^\eta x(t),\,^\eta x_t)\in\M_2$ is Malliavin differentiable for all $t\geq0$, then, for all $s\geq0$, $\DM_s\,^\eta x_t=\{\DM_s\,^\eta x(t+u),\,u\in[-r,0]\}$ and $\DM_s X^0_t(\eta)=(\DM_s\,^\eta x(t),\DM_s\,^\eta x(t+\cdot))\in \M_2$.
\end{lemma}

\begin{proof}
 The proof follows the same lines as the proof of Lemma \ref{DiffofSegmentequalsSegmentofDiff}.
\end{proof}

Here below we discuss the chain rule for the Malliavin derivative in $\M_2$. This leads to the study of the interplay between Malliavin derivatives and Fr\'echet derivatives.

We recall that, if $DX^0_T$ is bounded, i.e. for all $\omega=(\omega^W,\omega^\mathbb{B})\in\Omega$, $\sup_{\eta\in\M_2}\lVert DX^0_T(\eta(\omega),\omega^W)\rVert<\infty$, the chain rule in \cite[Proposition 3.8]{Pronk} gives
\begin{align*}
 \DM_sX^0_T(\eta(\omega^W,\omega^\mathbb{B}),\omega^W)=DX^0_T(\eta(\omega^W,\omega^\mathbb{B}),\omega^W)[\DM_s\eta(\omega^W,\omega^\mathbb{B})],
\end{align*}
as the Malliavin derivative only acts on $\omega^\mathbb{B}$. We need an analoguous result also in the case when $DX^0_T$ is possibly unbounded. To show this, we apply $\DM_s$ directly to the dynamics given by equation \eqref{delayeq0}.

\begin{theorem}\label{GeneralizedChainRule}
 Let $X^0_\cdot(\eta)\in L^2(\Omega;M_2([-r,T],\R^d))$ be the solution of \eqref{delayeq0}. Let Hypotheses \textbf{\emph{(EU)}}, \textbf{\emph{(Flow)}} and \textbf{\emph{(H)}} be fulfilled. Then we have
 \begin{align}
 \DM_sX^0_T(\eta)=DX^0_T(\eta)[\DM_s\eta]\quad(\omega,s)-a.e.
\end{align}
\end{theorem}

\begin{proof}
 To show this, we apply $\DM_s$ directly to the dynamics given by equation \eqref{delayeq0}. Doing this, we get, by definition of the operator $\rho_0$ and Lemma \ref{MalliavinofSegmentequalsSegmentofMalliavin}, for a.e. $\omega\in\Omega$

\begin{align}\label{MalliavinSDE}
\begin{split}
 \rho_0(\DM_sX^0_T(\eta))=\DM_s\ ^{\eta} x(t)=\begin{cases}\DM_s\eta(0) + \int_0^t Df(u,X^0_u(\eta))[\DM_sX^0_u(\eta)]du \\
+ \int_0^t Dg(u, X^0_u(\eta))[\DM_sX^0_u(\eta)]dW(u), \quad t\in[0,T], \\
\DM_s\eta(t), \quad t\in[-r,0].
\end{cases}
\end{split}
\end{align}

\noindent Define the processes
\begin{align*}
 y(t):=\begin{pmatrix}
         ^{\eta} x(t)\\
         D\ ^{\eta} x(t)[\DM_s\eta]
        \end{pmatrix},\ z(t):=\begin{pmatrix}
        ^{\eta} x(t)\\
        \DM_s\ ^{\eta} x(t)
       \end{pmatrix}.
\end{align*}
From the proof of Lemma \ref{4thMoments} we know that $y$ satisfies the SFDE
\begin{align*}
\begin{cases}
 y(t)&=
\begin{pmatrix}
        \eta(0)\\
        \DM_s\eta(0)
       \end{pmatrix} + \int_0^t \hat f(u, y(u), y_u)du + \int_0^t \hat g(u, y(u), y_u)dW(u),\\
       y_0&=(\eta,\DM_s\eta),
\end{cases}
\end{align*}
with the functions $\hat f$ and $\hat g$ as in the proof of Lemma \ref{4thMoments}. Moreover, by \eqref{MalliavinSDE} and Lemma \ref{MalliavinofSegmentequalsSegmentofMalliavin}, it holds that $z$ satisfies the SFDE
\begin{align*}
\begin{cases}
 z(t)&=
\begin{pmatrix}
        \eta(0)\\
        \DM_s\eta(0)
       \end{pmatrix} + \int_0^t \hat f(u, z(u), z_u)du + \int_0^t \hat g(u, z(u), z_u)dW(u),\\
       z_0&=(\eta,\DM_s\eta).
\end{cases}
\end{align*}

\noindent Comparing those two SFDEs, it follows that $y=z$ in $L^2(\Omega, M_2([-r,T],\R^d))$. Therefore,
\begin{align*}
 E\bigg[\int_0^T\lVert y_t-z_t\rVert_{M_2}^2dt\bigg]&=E\bigg[\int_0^T|y(t)-z(t)|^2+\int_{-r}^0|y(t+u)-z(t+u)|^2dudt\bigg]\\
 &\leq(1+r)\lVert y-z\rVert_{L^2(\Omega, M_2([-r,T],\R^d))}=0,
\end{align*}
which implies that $\lVert y_t-z_t\rVert_{M_2}=0$ for a.e. $(\omega,t)\in\Omega\times[0,T]$.

\end{proof}

We now introduce the randomization of the initial condition. For this we consider an $\R$-valued functional $\xi$ of $\mathbb{B}$, non-zero $P$-a.s. In particular, $\xi$ is a random variable independent of $W$. Choose $\xi$ to be Malliavin differentiable w.r.t. $\mathbb{B}$ with $\DM_s\xi\neq0$ for almost all $(\omega,s)$. Furthermore, let $\eta,\,h\in\M_2$ be random variables on $\Omega^W$, i.e. $\eta(\omega)=\eta(\omega^W)$, $h(\omega)=h(\omega^W)$. We write $\eta,\,h\in\M_2(\Omega^W)$, where $\M_2(\Omega^W)$ denotes the space of random variables in $\M_2$ that only depend on $\omega^W\in\Omega^W$. Here $\eta$ plays the role of the "true" (i.e. not randomized) initial condition and $h$ plays the role of the direction in which we later are going to differentiate. For simpler notation, we define $\tilde\eta:=\eta-h$.

\begin{corollary}\label{GeneralizedChainRule1}
 Let $X^0_\cdot(\et+\lxm)\in L^2(\Omega;M_2([-r,T],\R^d))$ be the solution of \eqref{delayeq0} with initial condition $\et+\lxm\in\M_2$, where $\lambda\in\R$. Let Hypotheses \textbf{\emph{(EU)}}, \textbf{\emph{(Flow)}} and \textbf{\emph{(H)}} be fulfilled. Then we obtain
 \begin{align}
\begin{split}
 \DM_sX^0_T(\et(\omega^W)+\lambda\xi(\omega^\mathbb{B})h(\omega^W))=DX^0_T(\et(\omega^W)+\lambda\xi(\omega^\mathbb{B})h(\omega^W))[\lambda\DM_s\xi(\omega^\mathbb{B}) h(\omega^W)]
\end{split}
\end{align}
 $(\omega,s)$-a.e. In short hand notation:
\begin{align}
 \DM_sX^0_T(\et+\lxm)=DX^0_T(\et+\lxm)[\lambda\DM_s\xi h].
\end{align}
\end{corollary}

We are now giving a derivative free representation of the expectation of the Fr\'echet derivative of $\Phi\circ X^0_T$ at $\eta$ in direction $h$ in terms of a Skorohod integral. This representation will later be used to get a representation for the derivative of $p(\eta)$ in direction $h$.

\begin{theorem}\label{thm1}
Let Hypotheses \textbf{\emph{(EU)}}, \textbf{\emph{(Flow)}} and \textbf{\emph{(H)}} be satisfied and let $\Phi$ be Fr\'echet differentiable. 
Furthermore, 
let $a\in L^2([-r,0],\R)$ be such that $\int_{-r}^0a(s)ds=1$. If $a(\cdot)\xi/\DM_\cdot\xi$ is Skorohod integrable and if the Skorohod integral below and its evaluation at $\lambda=\frac{1}{\xi}\in\R$ are well defined then following relation holds
\begin{align}\label{SkorohodRep}
 E[D(\Phi\circ X^0_T)(\eta)[h]]=-E\left[\left\{\delta\Big(\Phi(X^0_T(\et+\lxm))a(\cdot)\frac{\xi}{\DM_\cdot\xi}\Big)\right\}\Big|_{\lambda=\frac{1}{\xi}}\right].
\end{align}
\end{theorem}

\begin{proof}

First of all we can see that, by Theorem \ref{GeneralizedChainRule}, we have the relation
\begin{align*}
 \DM_sX^0_T(\et+\lxm)=DX^0_T(\et+\lxm)[\lambda\DM_s\xi h]\quad (\omega,s)-a.e.
\end{align*}

\noindent Multiplication with $\frac{\xi}{\DM_s\xi}$ yields
\begin{align}\label{TechnicalMalliavinFrechetEq}
 \frac{\xi}{\DM_s\xi}\DM_sX^0_T(\et+\lxm)=DX^0_T(\et+\lxm)[ h]\lambda\xi\quad (\omega,s)-a.e.
\end{align}
For the above, we recall that $\DM_s\xi\neq0$ a.e. Since the right-hand side in \eqref{TechnicalMalliavinFrechetEq} is defined $\omega$-wise, the evaluation at $\lambda=\frac{1}{\xi}$ yields $DX^0_T(\et+h)[h]$. Summarising, we have
\begin{align*}
 \Big\{\frac{\xi}{\DM_s\xi}\DM_sX^0_T(\et+\lxm)\Big\}\Big|_{\lambda=\frac{1}{\xi}}=DX^0_T(\et+\lxm)[h]\lambda\xi\Big|_{\lambda=\frac{1}{\xi}}=DX^0_T(\et+h)[h]=DX^0_T(\eta)[h]
\end{align*}
Multiplying with $1=\int_{-r}^0a(s)ds$ and applying the chain rule, together with the fact that $D\Phi(X^0_T(\eta))$ is defined pathwise, we obtain
\begin{align*}
 E[D(\Phi\circ X^0_T)(\eta)[h]]&=E\left[ D\Phi(X_T^0(\eta))DX^0_T(\eta)[h] \right]= E\left[ \int_{-r}^0D\Phi(X_T^0(\eta))DX^0_T(\eta)[h]a(s)ds \right]\\
 &= E\left[ \Big\{\int_{-r}^0D\Phi(X_T^0(\et+\lxm))\DM_sX^0_T(\et+\lxm)a(s)\frac{\xi}{\DM_s\xi}ds\Big\}\Big|_{\lambda=\frac{1}{\xi}} \right]\\
 &= E\left[ \Big\{\int_{-r}^0\DM_s\{\Phi(X_T^0(\et+\lxm))\}a(s)\frac{\xi}{\DM_s\xi}ds\Big\}\Big|_{\lambda=\frac{1}{\xi}} \right].
\end{align*}
The partial integration formula for the Skorohod integral yields
\begin{align*}
 E[D(\Phi\circ X^0_T)(\eta)[h]]&= E\left[ \Big\{\Phi(X_T^0(\et+\lxm))\delta\Big(a(\cdot)\frac{\xi}{\DM\xi}\Big)-\delta\Big(\Phi(X_T^0(\et+\lxm))a(\cdot)\frac{\xi}{\DM\xi}\Big)\Big\}\Big|_{\lambda=\frac{1}{\xi}} \right]\\
 &= E\left[ \Phi(X_T^0(\eta))\delta\Big(a(\cdot)\frac{\xi}{\DM\xi}\Big)-\Big\{\delta\Big(\Phi(X_T^0(\et+\lxm))a(\cdot)\frac{\xi}{\DM\xi}\Big)\Big\}\Big|_{\lambda=\frac{1}{\xi}} \right].
\end{align*}
The result follows now by independence of $\Phi(X_T^0(\eta))$, which is $\F^W$-measurable, and $\delta\Big(a(\cdot)\frac{\xi}{\DM\xi}\Big)$, which is $\F^{\mathbb{B}}$-measurable.
\end{proof}

\begin{remark}
As for a numerically tractable approximation of the stochastic integral in the above formula we refer to \cite[Section 3.1]{Nualart}.
\end{remark}

\begin{proposition}\label{ExistenceSkorohodEval}
 Define $u(s,\lambda):=\Phi(X^0_T(\et+\lxm))a(s)\frac{\xi}{\DM_s\xi}$, $s\in[-r,0]$, $\lambda\in\R$. Assume that the Skorohod integral $\delta(u(\cdot,\lambda))$ exists for all $\lambda\in\R$. If for all $\Lambda>0$ there exists a $C>0$ such that for all $\lambda_1,\lambda_2\in\overline{\supp\xi^{-1}}$, $|\lambda_1|,|\lambda_2|<\Lambda$:
\begin{align*}
 \|u(\cdot,\lambda_1)-u(\cdot,\lambda_2)\|^2_{L^2(\Omega\times[-r,0])}+\|\DM(u(\cdot,\lambda_1)-u(\cdot,\lambda_2))\|^2_{L^2(\Omega\times[-r,0]^2)}<C|\lambda_1-\lambda_2|^2,
\end{align*}
then the evaluation $\delta(u(\cdot,\lambda))|_{\lambda=\frac{1}{\xi}}$ is well defined.
\end{proposition}

\begin{proof}
The Skorohod integral $\delta(u(\cdot,\lambda))$ is an element of $L^2(\Omega, \R)$. From
\begin{align*}
 \|\delta(u(\cdot,\lambda))\|^2_{L^2(\Omega, \R)}\leq\|u(\cdot,\lambda)\|^2_{L^2(\Omega\times[-r,0], \R)}+\|\DM u(\cdot,\lambda)\|^2_{L^2(\Omega\times[-r,0],\R)}
\end{align*}
(see \cite[eq. (1.47) Proof of Prop. 1.3.1]{Nualart}), under the assumptions above and by means of Kolmogorov's continuity theorem, we can see that the process
\begin{align*}
 Z:\Omega\times\overline{\supp\xi^{-1}}\rightarrow L^2(\Omega, \R),\text{ }\lambda\mapsto\delta(u(\cdot,\lambda))
\end{align*}
has a continuous version. Applying this continuous version, the evaluation at the random variable $\frac{1}{\xi}$ is well defined:
\begin{align*}
 \delta(u(\cdot,\lambda))(\omega)|_{\lambda=\frac{1}{\xi}}:=Z(\omega,\lambda)|_{\lambda=\frac{1}{\xi}}:=Z(\omega,\frac{1}{\xi}(\omega)).
\end{align*}
Hence we conclude.
\end{proof}

\vspace{5mm}
\subsection{Representation formula for Delta under a suitable choice of the randomization}

A particularly interesting choice of randomization is $\xi=\exp(\mathbb{B}(1_{[-r,0]}))$, since in this case, $\DM_s\xi=\xi$ for all $s\in[-r,0]$ and

\begin{align}\label{SkorohodEstimation}
 &\|\delta(u(\cdot,\lambda_1))-\delta(u(\cdot,\lambda_2))\|^2_{L^2(\Omega)}\nonumber\\
 &\hspace{1cm}\leq\|a\|^2_{L^2([-r,0])}(\|\Phi(X^0_T(\et+\lIxm))-\Phi(X^0_T(\et+\lIIxm))\|^2_{L^2(\Omega)}\\
 &\hspace{1cm}\quad+\|\DM\{\Phi(X^0_T(\et+\lIxm))-\Phi(X^0_T(\et+\lIIxm))\}\|^2_{L^2(\Omega\times[-r,0])}).\nonumber
\end{align}

\noindent In this setup, let the following hypotheses be fulfilled:\\
\par
\noindent\textbf{Hypotheses (A):} $\Phi$ is globally Lipschitz with Lipschitz constant $L_\Phi$ and $C^1$. The Fr\'echet derivative $D\Phi$ is globally Lipschitz with Lipschitz constant $L_{D\Phi}$.\\
\par
A more general payoff function $\Phi$ will be considered in the next subsection.
Recall that $p(\eta)=E[\Phi(X^0_T(\eta))]$ and the sensitivity to the initial path, the Delta, in direction $h\in M_2$ is $\partial_hp(\eta):=\frac{d}{d\varepsilon} p(\eta + \varepsilon h)|_{\varepsilon=0}$.

\begin{lemma}\label{InterchangingExpectationAndDifferentiation0}
 Under Hypotheses \textbf{\emph{(EU)}}, \textbf{\emph{(Flow)}}, \textbf{\emph{(H)}} and \textbf{\emph{(A)}}, we have
 \begin{align*}
  \partial_hp(\eta)=E[D(\Phi\circ X^0_T)(\eta)[h]].
 \end{align*}
\end{lemma}

\begin{proof}
 By definition of the directional derivative, we have
 \begin{align*}
  \partial_hp(\eta)=\lim_{\varepsilon\rightarrow0}E\Big[\frac{1}{\varepsilon}(\Phi(X^0_T(\eta+\varepsilon h))-\Phi(X^0_T(\eta)))\Big]=\lim_{\varepsilon\rightarrow0}E[f_\varepsilon],
 \end{align*}
 where $f_\varepsilon(\omega)=\frac{1}{\varepsilon}(\Phi(X^0_T(\eta+\varepsilon h,\omega))-\Phi(X^0_T(\eta,\omega)))\rightarrow D(\Phi\circ X^0_T(\omega))(\eta)[h]$ a.s. since the Fr\'echet derivative of $\Phi\circ X^0_T$ in $\eta$ is defined for $\omega$-a.e. Moreover,
 \begin{align*}
  |f_\varepsilon(\omega)|&=\frac{|\Phi(X^0_T(\eta+\varepsilon h,\omega))-\Phi(X^0_T(\eta,\omega))|}{\varepsilon}\leq L_\Phi\frac{\lVert X^0_T(\eta+\varepsilon h,\omega)-X^0_T(\eta,\omega)\rVert_{M_2}}{\varepsilon}=:g_\varepsilon(\omega).
 \end{align*}
 So if we can find $g\in L^1(\Omega,P)$ s.t. $g_\varepsilon\rightarrow g$ in $L^1$-convergence as $\varepsilon\rightarrow0$, we would have that $f_\varepsilon\rightarrow D(\Phi\circ X^0_T)(\eta)[h]$ in $L^1$-convergence by Pratt's lemma (see \cite[Theorem 1]{Pratt}). This would conclude the proof.\\
 \par
 Observe that, by the continuity of the norm $\lVert\cdot\rVert_{M_2}$ and the $\omega$-wise Fr\'echet differentiability of $X^0_T$ in $\eta$, we have that
 \begin{align*}
  g_\varepsilon(\omega)\rightarrow L_\Phi\lVert DX^0_T(\eta,\omega)[h]\rVert_{M_2},\text{ }\omega\text{-a.e.}
 \end{align*}
 Let $g(\omega):=L_\Phi\lVert DX^0_T(\eta,\omega)[h]\rVert_{M_2}$. By Lemma \ref{4thMoments}, $g\in L^1(\Omega, \R)$. We apply Vitali's theorem (see \cite[Theorem 16.6]{Schilling}) to show that the convergence $g_\varepsilon\rightarrow g$ holds in $L^1$. This means that we have to prove that the family $\{g_\varepsilon\}_{\varepsilon\in(-\delta,\delta)}$ for some $\delta>0$ is uniformly integrable. To show that, we will proceed in two steps:
 \begin{enumerate}
  \item[(1)] Prove that $\lVert g_\varepsilon\rVert_{L^2(\Omega)}<K$ for some constant $K$ not depending on $\varepsilon$.
  \item[(2)] Show that this implies  that $\{g_\varepsilon\}_{\varepsilon\in(-\delta,\delta)}$ is uniformly integrable.
 \end{enumerate}
 \textit{Step (1):} By Lemma \ref{4thMoments}, it holds that for each fixed $\varepsilon\in(-\delta,\delta)\backslash\{0\}$, the function $s\mapsto E[(\frac{1}{\varepsilon}\lVert X^0_s(\eta+\varepsilon h,\omega)-X^0_s(\eta,\omega)\rVert_{M_2})^2]$ is integrable un $[0,T]$. Now, making use of Jensen's inequality, Fubini's theorem and the Burkholder-Davis-Gundy inequality,
 
 \begin{align*}
  &E\bigg[\Big(\frac{1}{\varepsilon}\lVert X^0_T(\eta+\varepsilon h)-X^0_T(\eta)\rVert_{M_2}\Big)^2\bigg]\\
  &\quad=E\bigg[\frac{1}{\varepsilon^2}\Big(\big|\varepsilon h(0)+\int_0^Tf(s,X^0_s(\eta+\varepsilon h))-f(s,X^0_s(\eta))ds\\
  &\quad\quad\quad+\int_0^Tg(s,X^0_s(\eta+\varepsilon h))-g(s,X^0_s(\eta))dW(s)\big|^2\\
  &\quad\quad+\int_{-r}^01_{(-\infty,0)}(T+u)|\varepsilon h(u)|^2du\\
  &\quad\quad+\int_{-r}^01_{[0,\infty)}(T+u)\big|\varepsilon h(0)+\int_0^{T+u}f(s,X^0_s(\eta+\varepsilon h))-f(s,X^0_s(\eta))ds\\
  &\quad\quad\quad+\int_0^{T+u}g(s,X^0_s(\eta+\varepsilon h))-g(s,X^0_s(\eta))dW(s)\big|^2du\Big)\bigg]\\
  &\quad\leq3|h(0)|^2+\frac{3T}{\varepsilon^2}\int_0^TE[|f(s,X^0_s(\eta+\varepsilon h))-f(s,X^0_s(\eta))|^2]ds\\
  &\quad\quad+\frac{3}{\varepsilon^2}\int_0^TE[|g(s,X^0_s(\eta+\varepsilon h))-g(s,X^0_s(\eta))|^2]ds+\int_{-r}^0|h(u)|^2du+3r|h(0)|^2\\
  &\quad\quad+\frac{3}{\varepsilon^2}\int_{-r}^01_{[0,\infty)}(T+u)\int_0^{T+u}(T+u)E[|f(s,X^0_s(\eta+\varepsilon h))-f(s,X^0_s(\eta))|^2]dsdu\\
  &\quad\quad+\frac{3}{\varepsilon^2}\int_{-r}^01_{[0,\infty)}(T+u)\int_0^{T+u}E[|g(s,X^0_s(\eta+\varepsilon h))-g(s,X^0_s(\eta))|^2]dsdu\\
  &\quad\leq3(1+r)\lVert h\rVert_{M_2}^2+(3+r)T\int_0^T\frac{1}{\varepsilon^2}E[|f(s,X^0_s(\eta+\varepsilon h))-f(s,X^0_s(\eta))|^2]ds\\
  &\quad\quad+(3+r)\int_0^T\frac{1}{\varepsilon^2}E[|g(s,X^0_s(\eta+\varepsilon h))-g(s,X^0_s(\eta))|^2]ds\\
  &\quad\leq3(1+r)\lVert h\rVert_{M_2}^2+(3+r)(L_g^2+TL_f^2)\int_0^TE\bigg[\Big(\frac{1}{\varepsilon}\lVert X^0_s(\eta+\varepsilon h)-X^0_s(\eta)\rVert_{M_2}\Big)^2\bigg]ds.
 \end{align*}
 It follows from Gr\"{o}nwall's inequality that
 \begin{align*}
  \lVert g_\varepsilon\rVert_{L^2(\Omega)}^2&=L_\Phi^2E\bigg[\Big(\frac{1}{\varepsilon}\lVert X^0_T(\eta+\varepsilon h)-X^0_T(\eta)\rVert_{M_2}\Big)^2\bigg]\leq3L_\Phi^2(1+r)\lVert h\rVert_{M_2}^2e^{(3+r)(TL_g^2+T^2L_f^2)}=:K^2.
 \end{align*}

 \textit{Step (2):} Fix $\delta>0$. Then, by H\"{o}lder's inequality and Markov's inequality
 \begin{align*}
  \lim_{M\rightarrow\infty}\sup_{|\varepsilon|<\delta}E[|g_\varepsilon|1_{\{|g_\varepsilon|>M\}}]&\leq\lim_{M\rightarrow\infty}\sup_{|\varepsilon|<\delta}\lVert g_\varepsilon\rVert_{L^2(\Omega)}\sqrt{P(|g_\varepsilon|>M)}\\
  &\leq\lim_{M\rightarrow\infty}\sup_{|\varepsilon|<\delta}\frac{\lVert g_\varepsilon\rVert_{L^2(\Omega)}^2}{M}\leq\lim_{M\rightarrow\infty}\frac{K^2}{M}=0,
 \end{align*}
 i.e. the family $\{g_\varepsilon\}_{\varepsilon\in(-\delta,\delta)}$ is uniformly integrable.
 
\end{proof}

With this result, we can give a derivative free representation formula for the directional derivatives of $p(\eta)$.

\begin{theorem}\label{thm1special}
 Let Hypotheses \textbf{\emph{(EU)}}, \textbf{\emph{(Flow)}}, \textbf{\emph{(H)}} and \textbf{\emph{(A)}} be fulfilled. Let $a\in L^2([-r,0],\R)$ be such that $\int_{-r}^0a(s)ds=1$ and let $\xi=\exp(\mathbb{B}(1_{[-r,0]}))$. Then the directional derivatives of $p$ have representation
\begin{align}\label{SkorohodRepSpecial}
 \partial_hp(\eta)=-E\left[\left\{\delta\Big(\Phi(X^0_T(\et+\lxm))a(\cdot)\Big)\right\}\Big|_{\lambda=\frac{1}{\xi}}\right].
\end{align}
\end{theorem}

To prove the theorem, we need the following lemma:

\begin{lemma}\label{technicalLemma1}
 Assume \textbf{\emph{(H)}} and \textbf{\emph{(A)}} and $\xi=\exp(\mathbb{B}(1_{[-r,0]}))$. For any $\Lambda>0$ there exists a $C>0$ such that, for all $|\lambda_1|,|\lambda_2|<\Lambda$, we have
 \begin{enumerate}
  \item [(i)] $E[\|X^0_T(\et+\lIxm)-X^0_T(\et+\lIIxm)\|^4_{M_2}]^{\frac{1}{2}}\leq C|\lambda_1-\lambda_2|^2$
  \item [(ii)] $E[\|DX^0_T(\et+\lIxm)[\lIxm]\|^4_{M_2}]^{\frac{1}{2}}\leq C|\lambda_1|^2$
  \item [(iii)] $E[\|DX^0_T(\et+\lIxm)[\lIxm]-DX^0_T(\et+\lIIxm)[\lIIxm]\|^2_{M_2}]\leq C|\lambda_1-\lambda_2|^2.$
 \end{enumerate}
\end{lemma}

\begin{proof}
 See Appendix.
\end{proof}

\begin{proof}[of Theorem \ref{thm1special}]
 By Lemma \ref{InterchangingExpectationAndDifferentiation0}, we know that we can interchange the directional derivative with the expectation. We shall prove that the Skorohod integral in \eqref{SkorohodRepSpecial} is well defined. For this we apply Proposition \ref{ExistenceSkorohodEval} and use \eqref{SkorohodEstimation}.
 
 Let $\lambda_1, \lambda_2\in\R$, $|\lambda_1|,|\lambda_2|<\Lambda$. Because of Hypotheses \textbf{(A)}, and by Lemma \ref{technicalLemma1}, we have that
 \begin{align*}
  \|\Phi(X^0_T(\et+\lIxm))-\Phi(X^0_T(\et+\lIIxm))\|^2_{L^2(\Omega)}&\leq L^2_\Phi E[\|X^0_T(\et+\lIxm)-X^0_T(\et+\lIIxm)\|^2_{M_2}]\\
  &\leq L_\Phi^2E[\|X^0_T(\et+\lIxm)-X^0_T(\et+\lIIxm)\|^4_{M_2}]^{\frac{1}{2}}\\
  &\leq L_\Phi^2C|\lambda_1-\lambda_2|^2.
 \end{align*}
 On the other hand, the chain rule for the Malliavin derivative, the property $\DM_s\xi=\xi$, the fact that for two linear operators $A_1$ and $A_2$ it holds $A_1x_1-A_2x_2=(A_1-A_2)x_1+A_2(x_1-x_2)$ together with the property $|a+b|^2\leq 2|a|^2+2|b|^2$ yield
 \begin{align*}
   &|\DM_s\{\Phi(X^0_T(\et+\lIxm))-\Phi(X^0_T(\et+\lIIxm))\}|^2\\
   &\leq2|(D\Phi(X^0_T(\et+\lIxm))-D\Phi(X^0_T(\et+\lIxm)))[DX^0_T(\et+\lIxm)[\lIxm]]|^2\\
   &\quad+2|D\Phi(X^0_T(\et+\lIIxm)[DX^0_T(\et+\lIxm)[\lIxm]-DX^0_T(\et+\lIIxm)[\lIIxm]]|^2\\
   &\leq2\|D\Phi(X^0_T(\et+\lIxm))-D\Phi(X^0_T(\et+\lIxm))\|^2\|[DX^0_T(\et+\lIxm)[\lIxm]]\|_{M_2}^2\\
   &\quad+2\|D\Phi(X^0_T(\et+\lIIxm)\|^2\|[DX^0_T(\et+\lIxm)[\lIxm]-DX^0_T(\et+\lIIxm)[\lIIxm]]\|_{M_2}^2\\
   &\leq2L^2_{D\Phi}\|X^0_T(\et+\lIxm)-X^0_T(\et+\lIxm)\|_{M_2}^2\|[DX^0_T(\et+\lIxm)[\lIxm]]\|_{M_2}^2\\
   &\quad+2L_\Phi^2\|[DX^0_T(\et+\lIxm)[\lIxm]-DX^0_T(\et+\lIIxm)[\lIIxm]]\|_{M_2}^2,
 \end{align*}
where we used Hypothesis \textbf{(A)} in the end. Taking expectations, applying H\"{o}lder's inequality and Lemma \ref{technicalLemma1} we finally get
\begin{align*}
 &\|\DM\{\Phi(X^0_T(\et+\lIxm))-\Phi(X^0_T(\et+\lIIxm))\}\|^2_{L^2(\Omega\times[-r,0])}\\
 &\leq2L^2_{D\Phi}E[\|X^0_T(\et+\lIxm)-X^0_T(\et+\lIxm)\|_{M_2}^4]^{\frac{1}{2}}E[\|[DX^0_T(\et+\lIxm)[\lIxm]]\|_{M_2}^4]^{\frac{1}{2}}\\
 &\quad+2L_\Phi^2E[\|[DX^0_T(\et+\lIxm)[\lIxm]-DX^0_T(\et+\lIIxm)[\lIIxm]]\|_{M_2}^2]\\
 &\leq2(L^2_{D\Phi}C^2|\lambda_1|^2+L_\Phi^2C)|\lambda_1-\lambda_2|^2\\
 &=\mathcal{O}(1)|\lambda_1-\lambda_2|^2.
\end{align*}
Hence, Proposition \ref{ExistenceSkorohodEval} guarantees the existence of the evaluation of the Skorohod integral in $\lambda=\frac{1}{\xi}$.
\end{proof}

\vspace{5mm}
\subsection{Generalization to a larger class of payoff functions}

Instead of Hypothesis \textbf{(A)}, assume now that the following holds:\\
\par

\noindent\textbf{Hypotheses (A'):} The payoff function $\Phi:M_2\rightarrow \R$ is convex, bounded from below and globally Lipschitz continuous with Lipschitz constant $L_\Phi$.\\
\par

\noindent Moreover, consider the \textit{Moreau-Yosida approximations} $\Phi_n:M_2\rightarrow\R$ given by

\begin{align}\label{Yosida}
 \Phi_n(x):=\inf_{y\in M_2}\Big(\Phi(y)+\frac{n}{2}\lVert x-y\rVert_{M_2}^2\Big).
\end{align}

\noindent The following lemma summarizes some well-known properties of the Moreau-Yosida approximations in our setup.

\begin{lemma}\label{PropertiesYosida}
 For $\Phi$ and $\Phi_n$ as above, the following holds
 \begin{enumerate}
 \item[(i)] $\Phi_n(x)=\Phi(J_n(x))+\frac{n}{2}\lVert x-J_n(x)\rVert_{M_2}^2$, $x\in M_2$, where $J_n$ is given by
 \begin{align*}
  n(x-J_n(x))\in\partial\Phi(J_n(x))\,\text{ or, equivalently }\,J_n=\left(\id+\frac{\partial\Phi}{n}\right)^{-1},
 \end{align*}
 where $\partial\Phi(x)$ denotes the subdifferential of $\Phi$ in $x$ and $\partial\Phi:=\{(x,y)\in M_2\times M_2: y\in\partial\Phi(x)\}$.
 \item[(ii)] For all $x\in M_2$, $\Phi_n(x)\uparrow\Phi(x)$ and $J_n(x)\rightarrow x$, as $n\rightarrow\infty$.
 \item[(iii)] $\Phi_n$ is Fr\'echet differentiable and, for all $x\in M_2$, it holds $$D\Phi_n(x)=n(x-J_n(x))\in\partial\Phi(J_n(x))$$ and $D\Phi_n$ is Lipschitz.
 \item[(iv)] For each point $x\in\text{dom}(\partial\Phi)$,
 \begin{align*}
  D\Phi_n(x)\rightarrow\partial^0\Phi(x),
 \end{align*}
 where $\partial^0\Phi(x)$ denotes the element $y\in\partial\Phi(x)$ with minimal norm.
  \item[(v)] For each $x\in M_2$, it holds $\lVert D\Phi_n(x)\rVert\leq L_\Phi$.
\end{enumerate}
\end{lemma}

\begin{proof}
 \textbf{(i):} See \cite[p. 58]{BorweinNoll} or \cite[Theorem 3.24, p. 301]{Attouch}, .\\
 \textbf{(ii):} See Theorem 2.64 in \cite[p. 229]{Attouch}.\\
 \textbf{(iii):} See \cite[p. 58]{BorweinNoll}, and \cite[Thm. 3.24]{Attouch}.\\
 \textbf{(iv):} See \cite[Proposition 3.56 (c), equation (3.136), p. 354]{Attouch}.\\
 \textbf{(v):} By (iii), it holds $D\Phi_n(x)\in\partial\Phi(y_0)$ for some $y_0\in M_2$ (namely $y_0=J_n(x)$). By the definition of the subdifferential, it holds for every $g\in\partial\Phi(y_0)$ and every $h\in M_2$:
 \begin{align*}
  \langle g,h\rangle\leq\Phi(y_0+h)-\Phi(y_0)\leq L_\Phi\lVert h\rVert_{M_2}.
 \end{align*}
 In particular, $D\Phi_n(x)[h]\leq L_\Phi\lVert h\rVert_{M_2}$ and $D\Phi_n(x)[-h]\leq L_\Phi\lVert h\rVert_{M_2}$ and thus
 \begin{align*}
  |D\Phi_n(x)[h]|\leq L_\Phi\lVert h\rVert_{M_2},\text{ which implies }\lVert D\Phi_n(x)\rVert\leq L_\Phi.\quad
 \end{align*}
\end{proof}

The following lemma shows that we can approximate $p(\eta)$ by a sequence $p_n(\eta)$ using the Moreau-Yosida approximations for the payoff functions.

\begin{proposition}
 Let the payoff function $\Phi:M_2\rightarrow \R$ be of type \textbf{\emph{(A')}}. Let $\Phi_n$ be given by \eqref{Yosida}. Set $p_n(\eta):=E[\Phi_n(X^0_T(\eta))]$ for $\eta\in\M_2$. Then, for all $\eta\in\M_2$, $p_n(\eta)\rightarrow p(\eta)$ as $n\rightarrow\infty$.
\end{proposition}

\begin{proof}
 As $\Phi$ is bounded from below, we can w.l.o.g. assume $\Phi$ being nonnegative. Then it is immediately clear from \eqref{Yosida} that also $\Phi_n$ is nonnegative for every $n$. Since $\Phi_n(x)\uparrow\Phi(x)$, we have that, for every $\omega\in\Omega$, $\Phi_n(X^0_T(\eta,\omega))\uparrow\Phi(X^0_T(\eta,\omega))$ and therefore, by monotone convergence
 \begin{align*}
  \lim_{n\rightarrow\infty}p_n(\eta)=\lim_{n\rightarrow\infty}E[\Phi_n(X^0_T(\eta))]=E[\Phi(X^0_T(\eta))]=p(\eta).\quad
 \end{align*}
\end{proof}

\begin{definition}
 Let $\mathcal X$ and $\mathcal Y$ be Banach spaces. We call a function $F:\mathcal X\rightarrow \mathcal Y$ \emph{LC directional differentiable} at $x\in \mathcal X$ if the directional derivative $\partial_hF(x)$ exists for each direction $h\in\mathcal X$ and defines a bounded linear operator from $\mathcal X$ to $\mathcal Y$. 
\end{definition}

\begin{lemma}\label{YosidaConvergenceDerivative}
 For each point $x\in M_2$ at which $\Phi$ is LC directional differentiable, it holds
 \begin{align*}
  D\Phi_n(x)\rightarrow \partial_\cdot\Phi(x).
 \end{align*}
\end{lemma}

\begin{proof}
 Since $\Phi$ is directional differentiable in $x$ in each direction $h\in M_2$, it holds that $\partial\Phi(x)$ is a singleton. In fact, by definition of the subdifferential and the directional derivative,

\begin{align*}
  \forall h\in M_2\,\begin{cases}
                     &\partial_h\Phi(x)=\lim_{\varepsilon\rightarrow0}\frac{\Phi(y_0+\varepsilon h)-\Phi(y_0)}{\varepsilon}\geq \langle g,h\rangle,\,\forall g\in\partial\Phi(x)\\
                     &\partial_h\Phi(x)=-\partial_{-h}\Phi(x)\leq -\langle g,-h\rangle=\langle g,h\rangle,\,\forall g\in\partial\Phi(x),
                    \end{cases}
 \end{align*}
 i.e. $\partial\Phi(x)=\{\partial_{\cdot}\Phi(x)\}$. It follows by Lemma \ref{PropertiesYosida} (iv) that $D\Phi_n(x)\rightarrow\partial^0\Phi(x)=\partial_{\cdot}\Phi(x)$.
\end{proof}

The following lemma, which is directly taken out of \cite{Phelps}, shows that the set of points where $\Phi$ is not LC directional differentiable, is a Gaussian null set. Recall that a measure $\mu$ on a Banach space $\mathcal{B}$ is called \textit{Gaussian} if for any nonzero $b\in\mathcal{B}^*$, the image measure $b_*(\mu):=\mu\circ b^{-1}$ is a Gaussian measure on $\R$. It is called \textit{nondegenerate}, if for any $b\in\mathcal{B}^*$, the variance of $b_*(\mu)$ is nonzero.

\begin{lemma}\label{GateauxNullset}
 Let $\mathcal{X}$ be a real separable Banach space, $\mathcal{Y}$ be a real Banach space such that every function $[0,1]\rightarrow\mathcal{Y}$ of bounded variation is a.e. differentiable, $\emptyset\neq G\subset\mathcal{X}$ open. Moreover, let $T: G\rightarrow\mathcal{Y}$ be a locally Lipschitz mapping. Then $T$ is LC directional differentiable outside a Gaussian null subset of $G$, i.e. for every nondegenerate Gaussian measure $\mu$ on $G$,
 \begin{align*}
  \mu(\{x\in G:\,T\text{ not LC directional differentiable in }x\})=0.
 \end{align*}
\end{lemma}

\begin{proof}
 See Theorem 1, Chapter 2 of \cite{Aronszajn} and Theorem 6 in \cite{Phelps}.
\end{proof}

\noindent This motivates the following assumption:\\
\par
\noindent\textbf{Hypothesis (G):} The distribution of $X^0_T(\eta)$ is absolutely continuous w.r.t. some nondegenerate Gaussian measure, namely it holds $P_{X^0_T(\eta)}:=X^0_T(\eta)(P):=P\circ (X^0_T(\eta))^{-1}\ll\mu$
for some nondegenerate Gaussian measure $\mu$.\\
\par

\noindent The following lemma provides a chain rule for $\Phi\circ X^0_T$

\begin{lemma}\label{ChainRulePhi}
 Let $\eta\in\M_2$ and $h\in M_2$. Under Hypotheses \textbf{\emph{(EU)}}, \textbf{\emph{(Flow)}}, \textbf{\emph{(H)}}, \textbf{\emph{(A')}} and \textbf{\emph{(G)}} it holds that the directional derivative $\partial_h(\Phi\circ X^0_T)(\eta)$ exists a.s. and we have
 \begin{align*}
 \partial_h(\Phi\circ X^0_T)(\eta)=\partial_{DX^0_T(\eta)[h]}\Phi(X^0_T(\eta)).
\end{align*}
\end{lemma}

\begin{proof}
 By definition of the directional derivative, we have
 \begin{align*}
  \partial_h(\Phi\circ X^0_T)(\eta)&=\lim_{\varepsilon\rightarrow0}\frac{\Phi(X^0_T(\eta+\varepsilon h))-\Phi(X^0_T(\eta))}{\varepsilon}\\
  &=\lim_{\varepsilon\rightarrow0}\bigg(\frac{\Phi\Big(X^0_T(\eta)+\varepsilon\frac{X^0_T(\eta+\varepsilon h)-X^0_T(\eta)}{\varepsilon}\Big)-\Phi(X^0_T(\eta)+\varepsilon DX^0_T(\eta)[h])}{\varepsilon}\\
  &\quad+\frac{\Phi(X^0_T(\eta)+\varepsilon DX^0_T(\eta)[h])-\Phi(X^0_T(\eta))}{\varepsilon}\bigg).
 \end{align*}
 Remark that, by Hypothesis \textbf{(A')}, $\Phi$ is Lipschitz, and, by Hypotheses \textbf{(EU)}, \textbf{(Flow)} and \textbf{(H)}, $X^0_T$ is Fr\'echet differentiable. Then we have for the first summand in this limit
 \begin{align*}
  &\bigg|\frac{\Phi\Big(X^0_T(\eta)+\varepsilon\frac{X^0_T(\eta+\varepsilon h)-X^0_T(\eta)}{\varepsilon}\Big)-\Phi(X^0_T(\eta)+\varepsilon DX^0_T(\eta)[h])}{\varepsilon}\bigg|\\
  &\quad\leq L_\Phi\bigg|\frac{X^0_T(\eta)+\varepsilon\frac{X^0_T(\eta+\varepsilon h)-X^0_T(\eta)}{\varepsilon}-X^0_T(\eta)-\varepsilon DX^0_T(\eta)[h]}{\varepsilon}\bigg|\\
  &\quad=L_\Phi\bigg|\frac{X^0_T(\eta+\varepsilon h)-X^0_T(\eta)}{\varepsilon}- DX^0_T(\eta)[h]\bigg|\rightarrow0,\quad\text{ as }\varepsilon\rightarrow0.
 \end{align*}
 As for the second summand in the above limit, by Hypothesis \textbf{(G)} and Lemma \ref{GateauxNullset}, we immediately have that
\begin{align*}
 P(\{\omega\in\Omega:\,\Phi\text{ is not LC directional differentiable in }X^0_T(\eta,\omega)\})=0
\end{align*}
and thus,
\begin{align*}
 \partial_{DX^0_T(\eta)[h]}\Phi(X^0_T(\eta))=\lim_{\varepsilon\rightarrow0}\frac{\Phi(X^0_T(\eta)+\varepsilon DX^0_T(\eta)[h])-\Phi(X^0_T(\eta))}{\varepsilon}
\end{align*}
exists almost surely. This ends the proof.
\end{proof}

\begin{proposition}
 Under Hypotheses \textbf{\emph{(EU)}}, \textbf{\emph{(Flow)}}, \textbf{\emph{(H)}}, \textbf{\emph{(A')}} and \textbf{\emph{(G)}} it holds
 \begin{align}
  \partial_hp_n(\eta)\rightarrow\partial_hp(\eta).
 \end{align}
\end{proposition}

\begin{proof}
By Lemma \ref{GateauxNullset} and Hypothesis \textbf{(G)}, we have that
\begin{align*}
 P(\{\omega\in\Omega:\,\Phi\text{ is not LC directional differentiable in }X^0_T(\eta,\omega)\})=0,
\end{align*}
and thus, by Lemma \ref{YosidaConvergenceDerivative},
\begin{align*}
 D\Phi_n(X^0_T(\eta)) \rightarrow \partial_\cdot\Phi(X^0_T(\eta)),\text{ a.s.}
\end{align*}
Therefore, applying the Fr\'echet differentiability of the mapping $\eta\mapsto X^0_T(\eta)$, the chain rule from Lemma \ref{ChainRulePhi} and the fact that the LC directional derivative is a continuous linear mapping (in the direction), we obtain

\begin{align*}
 |D(\Phi_n\circ X^0_T)(\eta)[h]-\partial_h(\Phi\circ X^0_T)(\eta)|&=|D\Phi_n(X^0_T(\eta))DX^0_T(\eta)[h]-\partial_{DX^0_T(\eta)[h]}\Phi(X^0_T(\eta))|\\
 &=|(D\Phi_n(X^0_T(\eta))-\partial_\cdot\Phi(X^0_T(\eta)))[DX^0_T(\eta)[h]]|\\
 &\leq\lVert D\Phi_n(X^0_T(\eta))-\partial_\cdot\Phi(X^0_T(\eta))\rVert\cdot\lVert DX^0_T(\eta)\rVert\cdot\lVert h\rVert\\
 &\rightarrow0,\text{ a.s., as }n\rightarrow\infty.
\end{align*}
Moreover, by Lemma \ref{PropertiesYosida} (v) and Lemma \ref{4thMoments}, it holds
\begin{align*}
 |D(\Phi_n\circ X^0_T)(\eta)[h]|&\leq\lVert D\Phi_n(X^0_T(\eta))\rVert\cdot\lVert DX^0_T(\eta)[h]\rVert\leq L_\Phi\lVert DX^0_T(\eta)[h]\rVert\in L^1(\Omega).
\end{align*}

Furthermore, similarly to the proof of Lemma \ref{InterchangingExpectationAndDifferentiation0}, it can be shown that

\begin{align}
  \partial_hp(\eta)&=E[\partial_h(\Phi\circ X^0_T)(\eta)]\text{ and}\label{InterchangingExpectationAndDifferentiation1}\\
  \partial_hp_n(\eta)&=E[D(\Phi_n\circ X^0_T)(\eta)[h]]\label{InterchangingExpectationAndDifferentiation2},
\end{align}

\noindent where, for \eqref{InterchangingExpectationAndDifferentiation1}, we use that the LC directional derivative of $\Phi\circ X^0_T$ is defined for a.e. $\omega\in\Omega$ (rather than the Fr\'echet derivative). It now follows by dominated convergence that
\begin{align*}
 \partial_hp_n(\eta)=E[D(\Phi_n\circ X^0_T)(\eta)[h]]\rightarrow E[\partial_h(\Phi\circ X^0_T)(\eta)]=\partial_hp(\eta).
\end{align*}
By this we end the proof. 
\end{proof}

Our final theorem summarizes the results of this section and shows that our representation formula \eqref{SkorohodRepSpecial} can be used in an approximation scheme for the directional derivatives of $p$ in this more general setup:

\begin{theorem}
 Let Hypotheses \textbf{\emph{(EU)}}, \textbf{\emph{(Flow)}}, \textbf{\emph{(H)}}, \textbf{\emph{(A')}} and \textbf{\emph{(G)}} be fulfilled. Let $\Phi_n$ denote the $n$th Moreau-Yosida approximation of $\Phi$. Then, for $\xi=\exp(\mathbb{B}(1_{[-r,0]}))$,
 \begin{align}
  \partial_hp(\eta)=-\lim_{n\rightarrow\infty}E\left[\left\{\delta\Big(\Phi_n(X^0_T(\et+\lxm))a(\cdot)\Big)\right\}\Big|_{\lambda=\frac{1}{\xi}}\right].
 \end{align}
\end{theorem}

\begin{proof}
 As we have shown so far, $\partial_hp(\eta)=\lim_{n\rightarrow\infty}E[D(\Phi_n\circ X^0_T)(\eta)[h]]$. It follows from Lemma \ref{PropertiesYosida} (iii) and (v) that $\Phi_n$ satisfies Hypothesis \textbf{(A)}. Therefore, we can apply Theorem \ref{thm1special}.
\end{proof}

\begin{remark}
 Making use of the linearity of the derivative operator and the expectation, this result can easily be generalised to $\Phi$ being given by the difference of two convex, bounded from below and globally Lipschitz continuous functions $\Phi^{(1)}$ and $\Phi^{(2)}$.
\end{remark}

To conclude this section, we provide an example, where the Hypothesis \textbf{(G)} holds.

\begin{example}
 Let $d=m$, $T>r$, $f$ be bounded and $g(s,\varphi)=Id_{d\times d}$, i.e.
 \begin{align*}
 \begin{cases}
  ^\eta x(t)&=\eta(0)+\int_0^tf(s,\,^\eta x(s),\,^\eta x_s)ds+W(t),\,t\in[0,T]\\
  ^\eta x_0&=\eta.
 \end{cases}
 \end{align*}
 Then, application of Girsanov's theorem (Novikov's condition is satisfied) yields that
 \begin{align*}
  ^\eta \tilde{W}(t):=\int_0^tf(s,\,^\eta x(s),\,^\eta x_s)ds+W(t)
 \end{align*}
 is an $m$-dimensional Brownian motion under a measure $^\eta Q\sim P$. Since $T>r$, we have
 \begin{align*}
  X^0_T(\eta)=(\eta(0)+\,^\eta \tilde{W}(T),\eta(0)+\,^\eta \tilde{W}_T).
 \end{align*}
 Now, since $P\ll\, ^\eta Q$, it holds also
 \begin{align*}
  P_{X^0_T(\eta)}\ll\, ^\eta Q_{X^0_T(\eta)}=\,^\eta Q_{(\eta(0)+^\eta \tilde{W}(T),\eta(0)+^\eta \tilde{W}_T)}.
 \end{align*}
 But $^\eta Q_{(\eta(0)+^\eta \tilde{W}(T),\eta(0)+^\eta \tilde{W}_T)}$ is a  Gaussian measure on $M_2$ as for every $e\in M_2$ and every $A\in\mathcal{B}(\R)$
 \begin{align*}
  &^\eta Q_{\langle(\eta(0)+\,^\eta \tilde{W}(T),\eta(0)+\,^\eta \tilde{W}_T),e\rangle}(A)=\,^\eta Q(\langle(\eta(0)+\,^\eta \tilde{W}(T),\eta(0)+\,^\eta \tilde{W}_T),e\rangle\in A)\\
  &\quad=\,^\eta Q\bigg(\eta(0)\Big(e(0)+\int_{-r}^0e(u)du\Big)+\,^\eta \tilde{W}(T)e(0)+\int_{-r}^0\,^\eta \tilde{W}(T+u)e(u)du\in A\bigg)
 \end{align*}
 and $^\eta \tilde{W}$ is a Gaussian process under $^\eta Q$.
\end{example}

\vspace{5mm}
\section*{Appendix}

\noindent
\emph{Proof of Lemma \ref{technicalLemma1}:}

 \textbf{(i):}
  \begin{align*}
   &E[\|X^0_T(\et+\lIxm)-X^0_T(\et+\lIIxm)\|^4_{M_2}]\\
   &=E\left[\left(|\ ^{\et+\lIxm}x(T)-\ ^{\et+\lIIxm}x(T)|^2_{\R^d}+\int_{T-r}^T |\ ^{\et+\lIxm}x(t)-\ ^{\et+\lIIxm}x(t)|^2_{\R^d}dt\right)^2\right].
  \end{align*}
 Now splitting up the integral into an integral on $[T-r,T-r\vee0]$ and an integral on $[T-r\vee0,T]$ as we have done already in the proof of Lemma \ref{4thMoments}, we get
 \begin{align*}
  \int_{T-r}^T |\ ^{\et+\lIxm}x(t)-\ ^{\et+\lIIxm}x(t)|^2_{\R^d}dt&\leq r|\lambda_1-\lambda_2|^2|\xi|^2\|h\|^2_{M_2}+\int_0^T |\ ^{\et+\lIxm}x(t)-\ ^{\et+\lIIxm}x(t)|^2_{\R^d}dt,
 \end{align*}
  and therefore,
 \begin{align*}
   E[\|X^0_T(\et+\lIxm)-X^0_T(\et+\lIIxm)\|^4_{M_2}]&\leq \OO(1)\bigg(E\left[|\ ^{\et+\lIxm}x(T)-\ ^{\et+\lIIxm}x(T)|^4_{\R^d}\right]+|\lambda_1-\lambda_2|^4\\
   &\quad+E\left[\int_{0}^T |\ ^{\et+\lIxm}x(t)-\ ^{\et+\lIIxm}x(t)|^4_{\R^d}dt\right]\bigg).
 \end{align*}
 Now consider the term $E\left[|\ ^{\et+\lIxm}x(t)-\ ^{\et+\lIIxm}x(t)|^4_{\R^d}\right]$. Similarly to the steps in the proof of Lemma \ref{4thMoments} (applying Jensen's inequality, Burkholder-Davis-Gundy's inequality and the Lipschitzianity of $f$ and $g$), we show that
  \begin{align*}
   &E\left[|\ ^{\et+\lIxm}x(t)-\ ^{\et+\lIIxm}x(t)|^4_{\R^d}\right]\\
   &\leq\OO(1)\bigg(|\lambda_1-\lambda_2|^4+(L_f^4+L_g^4)\int_0^TE[\|X^0_u(\et+\lIxm)-X^0_u(\et+\lIIxm)\|^4_{M_2}]du\bigg).
  \end{align*}
 Finally, we can plug this into the inequality from before and get
 \begin{align*}
   &E[\|X^0_T(\et+\lIxm)-X^0_T(\et+\lIIxm)\|^4_{M_2}]\\
   &\leq \OO(1)\bigg(|\lambda_1-\lambda_2|^4+(L_f^4+L_g^4)\int_0^TE[\|X^0_u(\et+\lIxm)-X^0_u(\et+\lIIxm)\|^4_{M_2}]du\\
   &\quad+\int_{0}^T |\lambda_1-\lambda_2|^4+(L_f^4+L_g^4)\int_0^TE[\|X^0_u(\et+\lIxm)-X^0_u(\et+\lIIxm)\|^4_{M_2}]dudt\bigg)\\
   &\leq \OO(1)\bigg(|\lambda_1-\lambda_2|^4+\int_0^TE[\|X^0_u(\et+\lIxm)-X^0_u(\et+\lIIxm)\|^4_{M_2}]du\bigg).
 \end{align*}
 Since we already know from Lemma \ref{4thMoments} that $t\mapsto E[\|X^0_t(\et+\lIxm)-X^0_t(\et+\lIIxm)\|^4_{M_2}]$ is integrable on $[0,T]$, the result follows directly by application of Gr\"{o}nwall's inequality and taking the square root.\\
 \par
 \textbf{(ii) and (iii):} The proof follows from the same considerations that we made in (i) and in the proof of Lemma \ref{4thMoments}, by applying Gr\"{o}nwall's inequality and make use of the fact that we have integrability of the functions $t\mapsto E[\|DX^0_T(\et+\lIxm)[\lIxm]\|^4_{M_2}]^{\frac{1}{2}}\leq C|\lambda_1|^2$ and $t\mapsto E[\|DX^0_T(\et+\lIxm)[\lIxm]-DX^0_T(\et+\lIIxm)[\lIIxm]\|^2_{M_2}]\leq C|\lambda_1-\lambda_2|^2$ by Lemma \ref{4thMoments}.
 \par

\vspace{5mm}
\noindent\textbf{Acknowledgements:} This research is conducted within the projects FINEWSTOCH (239019) and STOCHINF (250768) of the Research Council of Norway (NFR). The support of NFR is thankfully acknowledged.

\bibliographystyle{siam}
\bibliography{referances} 
%
%
%

\end{document}